\documentclass[11pt]{amsart}
\usepackage{amsmath,amsthm,amssymb,enumerate}
\usepackage{setspace}
\newtheorem{theorem}{Theorem}[section]
\newtheorem*{theorem*}{Theorem}
\newtheorem{lemma}[theorem]{Lemma}
\newtheorem*{lemma*}{Lemma}

\newtheorem*{example*}{Example}

\newtheorem*{corollary*}{Corollary}
\newtheorem{definition}[theorem]{Definition}
\newtheorem*{definition*}{Definition}
\newtheorem{lemma-definition}[theorem]{Lemma-Definition}
\newtheorem*{lemma-definition*}{Lemma-Definition}
\newtheorem{definition-remark}[theorem]{Definition-Remark}
\newtheorem*{definition-remark*}{Definition-Remark}
\newtheorem{remark}[theorem]{Remark}
\newtheorem*{remark*}{Remark}

\newtheorem*{notation*}{Notation}
\newtheorem{proposition}[theorem]{Proposition}
\newtheorem*{proposition*}{Proposition}
\newtheorem{lemma-notation}[theorem]{Lemma-Notation}

\begin{document}
\setlength{\baselineskip}{1.2\baselineskip}

\title{Quantization of Poisson-CGL extensions}
\author{Yipeng Mi}
\address{
Department of Mathematics   \\
The University of Hong Kong\\
Pokfulam Road, Hong Kong}
\email{littlemi@connect.hku.hk}
\date{}
\begin{abstract} CGL extensions, named after G. Cauchon, K. Goodearl, and E. Letzter, are a special class of noncommutative algebras that are iterated Ore extensions of associative algebras with compatible torus actions. Examples of CGL extensions include quantum Schubert cells and quantized coordinate rings of double Bruhat cells. CGL extensions have been studied extensively in connection with quantum groups and quantum cluster algebras. For a field $\mathbf{k}$ of characteristic $0$, let $L=\mathbf{k}[q^{\pm 1}]$ be the $\mathbf{k}$-algebra of Laurent polynomials in the single variable $q$ and let $\mathbb{K}=\mathbf{k}(q)$ be the fraction field of $L$. We introduce quantum-CGL extensions as certain $L$-forms of CGL extensions over $\mathbb{K}$, which have Poisson-CGL extensions as their semiclassical limits. Poisson-CGL extensions, recently introduced and systematically studied by K. Goodearl and M. Yakimov, are certain Poisson polynomial algebras which admit presentations as iterated Poisson-Ore extensions with compatible torus actions. Examples of Poisson-CGL extensions include the coordinate rings of matrix affine Poisson spaces and more generally those of Schubert cells. We describe an explicit procedure for constructing a symmetric quantum-CGL extension from a symmetric integral Poisson-CGL extension and establish the uniqueness of such a quantization in a proper sense.

\end{abstract}
\maketitle
\hfill

\section{Introduction}\label{sec-int}

\subsection{Introduction}\label{subsec-int}

In this paper, we explicitly quantize a special class of Poisson polynomial algebras. We first make precise the notion of quantization used in this paper.

Fix a field $\mathbf{k}$ of characteristic $0$ and let
\[ L=\mathbf{k}[q^{\pm 1}]=\mathbf{k}[q,q^{-1}] \]
be the $\mathbf{k}$-algebra of Laurent polynomials in the single variable $q$. Recall (see, for example, \cite[III.5.4]{BG} and \cite[$\S$3.1]{CP}) that if $\mathcal{A}$ is an $L$-algebra such that $q-1 \in \mathcal{A}$ is not a zero divisor and if the quotient $\mathcal{A}/(q-1)\mathcal{A}$ is commutative, then the $\mathbf{k}$-algebra $\mathcal{A}/(q-1)\mathcal{A}$ has a Poisson bracket $\{ \ , \ \}$ given by
\[ \{a+(q-1)\mathcal{A}, \ b+(q-1)\mathcal{A}\}=\frac{ab-ba}{q-1}+(q-1)\mathcal{A}, \ \ a,b \in \mathcal{A}. \]
The Poisson $\mathbf{k}$-algebra
\[ (\mathcal{A}/(q-1)\mathcal{A}, \ \{ \ , \ \}) \]
is called the semiclassical limit of $\mathcal{A}$. The following definition of quantization is sufficient for the purpose of this paper.

\begin{definition}
\label{qua-alg}
{\rm
1) By a quantum algebra over $L$, or a quantum $L$-algebra, we mean a unital $L$-algebra $A$ that is a free $L$-module and
such that $\mathcal{A}/(q-1)\mathcal{A}$ is commutative;

2) Given a Poisson $\mathbf{k}$-algebra $(A, \ \{ \ , \ \})$, by an $L$-quantization of $A$ we mean a quantum $L$-algebra $\mathcal{A}$ together with a Poisson $\mathbf{k}$-algebra isomorphism
\[ (\mathcal{A}/(q-1)\mathcal{A}, \ \{ \ , \ \}) \longrightarrow (A, \ \{ \ , \ \}). \]
\hfill $\diamond$
}
\end{definition}

For any finite dimensional Lie algebra $\mathfrak{g}$ over $\mathbf{k}$, the universal enveloping algebra $U(\mathfrak{g})$ of
$\mathfrak{g}$ can be made into an $L$-quantization of the Poisson algebra $S(\mathfrak{g})$ (see \cite[$\S$2.1]{De}). Other well-known examples include $L$-quantizations of the standard multiplicative Poisson structures on a semi-simple algebraic group $G$ and on its dual Poisson group, as proven in \cite[Theorem 6.1]{De} and \cite[Theorem 3.1]{De} respectively.

The Poisson polynomial $\mathbf{k}$-algebras that we quantize can be presented as Poisson-CGL extensions. The term "Poisson-CGL extensions" is introduced by K. Goodearl and M. Yakimov in \cite{PC}, and they are special examples of iterated Poisson-Ore extensions \cite{Oh} with compatible torus actions. See $\S$\ref{sec-Poi-CGL} for more details. Prime ideals of Poisson-CGL extensions are studied by K. Goodearl and S. Launois in \cite{DM-p}. On the other hand, cluster algebra structures on symmetric Poisson-CGL extensions are systematically studied in \cite{PC}. Typical examples of Poisson-CGL extensions include the coordinate rings of matrix affine Poisson spaces \cite{BGY} and many other Poisson algebras arising as the semiclassical limits of quantized coordinate rings. A family of Poisson-CGL extensions come from Bott-Samelson varieties \cite{EL}, which are the motivating examples of this paper. 

The quantizations we construct have presentations as quantum-CGL extensions, the definition of which is given in $\S$\ref{subsec-qua-CGL}. The quantization procedure is summarized in $\S$\ref{subsec-pro}, while details are given in $\S$\ref{subsec-qua}. Quantum-CGL extensions are certain $L$-forms of CGL extensions over the fraction field $\mathbb{K} = \mathbf{k}(q)$ of $L$. CGL extensions over arbitrary fields, first introduced and studied in \cite{LLR}, form a special class of noncommutative unique factorization domains. See $\S$\ref{sec-CGL} for more details. One may apply to CGL extensions the method of deleting derivations by G. Cauchon \cite{Cauchon} and the stratification theory of torus-invariant prime ideals by K. Goodearl and E. Letzter \cite{DM-q}. Recently, quantum cluster algebra structures on CGL extensions have been studied extensively by K. Goodearl and M. Yakimov \cite{QQ} \cite{QO} \cite{QQ2}. Quantum Schubert cell algebras, introduced by C. De Concini, V. Kac, C. Procesi \cite{DKP} and G. Lusztig \cite{Lusztig}, have presentations as CGL extensions \cite{QQ} \cite[Example 4.5]{QO} \cite[$\S$9.2]{QQ2}. In \cite[$\S$4]{DB}, quantized coordinate rings of double Bruit cells are shown to have presentations as CGL extensions as well. See \cite[Corollary 3.8]{LLR} for more examples.

\subsection{Main results}\label{subsec-mai}

Given a Poisson algebra $(A, \{ \ , \ \})$ over $\mathbf{k}$, recall from \cite{DM-p} \cite{Oh} that a Poisson-Ore extension of $A$ is the algebra $A[x]$ together with a Poisson bracket $\{ \ , \ \}$ which extends the Poisson bracket on $A$ and satisfies
\begin{equation}\label{eq-Ax}
\{x, \ A\} \subset Ax + A.
\end{equation}
Given such an extension, the two $\mathbf{k}$-linear maps $\theta$ and $\delta$ on $A$ given by
\begin{equation}\label{eq-ax}
\{x, \ a\} = \theta(a)x + \delta(a), \ \ a \in A,
\end{equation}
are necessarily derivations of $A$, and they have the additional properties that
$\theta$ is a Poisson derivation of $A$ and $\delta$ is a Poisson $\theta$-derivation of $A$, i.e., for $a_1, a_2 \in A$,
\begin{align}\label{eq-theta-delta-1}
\theta(\{a_1, \ a_2\}) &= \{\theta(a_1), \ a_2\} + \{a_1, \, \theta(a_2)\},\\
\label{eq-theta-delta-2}
\delta(\{a_1, \ a_2\}) &= \{\delta(a_1), \ a_2\} + \{a_1, \, \delta (a_2)\} + \theta(a_1) \delta(a_2) - \delta(a_1) \theta(a_2).
\end{align}
Conversely (see \cite{DM-p} \cite{Oh}), given derivations $\theta, \delta$ of $A$ satisfying (\ref{eq-theta-delta-1}) and (\ref{eq-theta-delta-2}),
the Poisson bracket on $A$ extends, via (\ref{eq-ax}) and the Leibniz rule, to a unique Poisson bracket on $A[x]$ satisfying (\ref{eq-Ax}). The polynomial algebra $A[x]$ with the Poisson structure so defined using $\theta$ and $\delta$ is denoted by $A[x; \theta, \delta]$.

In this paper, we quantize Poisson polynomial algebras that have presentations as symmetric integral Poisson-CGL extensions (see Definition \ref{T-Poi-CGL}, Definition \ref{integral}, and Definition \ref{P-sym}), which can be defined as follows:

\begin{definition}
\label{sym-int-Poi-CGL}
{\rm
A symmetric integral Poisson-CGL extension is an iterated Poisson-Ore extension
\[ A=\mathbf{k}[x_1;\theta_1,\delta_1] \cdots [x_n;\theta_n,\delta_n], \]
together with an action of a split $\mathbf{k}$-torus $\mathbb{T}$ by Poisson automorphisms, satisfying

1) $\delta_j(x_i) \in \mathbf{k}[x_{i+1},...,x_{j-1}]$ for $1 \leq i < j \leq n$;

2) $x_1,...,x_n$ are $\mathbb{T}$-weight vectors with weights $\lambda_1,...,\lambda_n$ respectively;

3) there exist $h_1,...,h_n,h'_1,...,h'_n$ in the co-character lattice of $\mathbb{T}$ such that

\indent \indent a) $\theta_j(x_i)=\lambda_i(h_j)x_i=-\lambda_j(h'_i)x_i$ for $1 \leq i < j \leq n$,

\indent \indent b) $\lambda_j(h_j) \neq 0$ and $\lambda_j(h'_j) \neq 0$ for $j \in [1,n]$.\\
As the $\mathbb{T}$-action on $A$ is determined by $\lambda_1, \ldots, \lambda_n$, we also denote the Poisson algebra $A$ by
\[
A = \mathbf{k}[x_1;\theta_1,\delta_1]\cdots
[x_n;\theta_n,\delta_n]_{(\lambda_1, \ldots,
\lambda_n)}.
\]
\hfill $\diamond$
}
\end{definition}

Recall that an Ore extension of a ring $R$ is the ring $R[x]$ generated by $R$ and $x$ that satisfies
\[ xr=\sigma(r)x+\Delta(r), \ \ r \in R, \]
where $\sigma$ is a ring endomorphism on $R$ and $\Delta$ is a $\sigma$-derivation, i.e.,
\[ \Delta(r_1r_2)=\sigma(r_1)\Delta(r_2)+\Delta(r_1)r_2, \ \ r_1,r_2 \in R. \]
The ring $R[x]$ is denoted as $R[x;\sigma,\Delta]$. A quantum-Ore extension of a quantum $L$-algebra $\mathcal{A}$ (see Definition \ref{qua-alg}) is an Ore extension $\mathcal{A}[X;\sigma,\Delta]$ of $\mathcal{A}$ such that $\sigma,\Delta$ are $L$-linear maps and $\mathcal{A}[X;\sigma,\Delta]$ is a quantum $L$-algebra.

The quantum $L$-algebras that we construct have presentations as symmetric quantum-CGL extensions (see Definition \ref{Q-CGL} and Definition \ref{Q-sym}), which can be defined as follows:

\begin{definition}
\label{sym-qua-CGL}
{\rm
A symmetric quantum-CGL extension is an iterated quantum-Ore extension
\[ \mathcal{A}=L[X_1;\sigma_1,\Delta_1]\cdots [X_n;\sigma_n,\Delta_n], \]
together with an action of a split $\mathbf{k}$-torus $\mathbb{T}$ by $\mathbf{k}$-algebra automorphisms, satisfying 

1) $\Delta_j(X_i)$ lies in the subalgebra of $\mathcal{A}$ generated by $X_{i+1},...,X_{j-1}$, for $1 \leq i < j \leq n$;

2) $X_1,...,X_n$ are $\mathbb{T}$-weight vectors with weights $\lambda_1,...,\lambda_n$ respectively;

3) there exist $h_1,...,h_n,h'_1,...,h'_n$ in the co-character lattice of $\mathbb{T}$ such that

\indent \indent a) $\sigma_j(X_i)=q^{\lambda_i(h_j)}X_i=q^{-\lambda_j(h'_i)}X_i$ for $1 \leq i < j \leq n$,

\indent \indent b) $\lambda_j(h_j) \neq 0$ and $\lambda_j(h'_j) \neq 0$ for $j \in [1,n]$.\\
As the $\mathbb{T}$-action on $\mathcal{A}$ is determined by $\lambda_1, \ldots, \lambda_n$, we also denote the quantum $L$-algebra $\mathcal{A}$ by
\[ \mathcal{A}=L[X_1;\sigma_1,\Delta_1]\cdots [X_n;\sigma_n,\Delta_n]_{(\lambda_1, \ldots, \lambda_n)}. \]
\hfill $\diamond$
}
\end{definition}

The following is the main theorem of this paper.

\begin{theorem}
\label{quantization}
Every symmetric integral Poisson-CGL extension has an $L$-quantization into a symmetric quantum-CGL extension.
\end{theorem}

Moreover, the quantization we construct is "preferred" (see Definition \ref{per-qua}) and preferred quantization is unique in the sense of Theorem \ref{unique}. We have an explicit procedure for this quantization, which is discussed in $\S$\ref{subsec-pro}.

\subsection{Quantization procedure}\label{subsec-pro}

In this section, we briefly describe the quantization procedure. See $\S$\ref{subsec-qua} for more details. Let
\[ A = \mathbf{k}[x_1;\theta_1,\delta_1]\cdots [x_n;\theta_n,\delta_n]_{(\lambda_1, \ldots, \lambda_n)} \]
be a symmetric integral Poisson-CGL extension (see Definition \ref{sym-int-Poi-CGL}).
For $1 \leq i < j \leq n$, let $\lambda_{i,j} \in \mathbf{k}$ be
such that $\theta_j(x_i)=\lambda_{i,j}x_i$. By 3) of Definition \ref{sym-int-Poi-CGL}, $\lambda_{i,j} \in \mathbb{Z}$.

We start with the polynomial algebra $L[X_n]$, which is clearly an $L$-quantization of the Poisson subalgebra $\mathbf{k}[x_n]$ of $A$,
and we proceed by induction. Assume that $n \geq 2$, and assume that for some $k \in [1,n-1]$
we have constructed an $L$-quantization $\mathcal{A}_{[k+1,n]}$ of $\mathbf{k}[x_{k+1},...,x_n]$. Next we construct a quantum-Ore extension $\mathcal{A}_{[k+1,n]}[X_k;\sigma_{[k]},\Delta_{[k]}]$ of $\mathcal{A}_{[k+1,n]}$. The map $\sigma_{[k]}$ is given by
\[ \sigma_{[k]}(X_j)=q^{-\lambda_{k,j}}X_j, \ \ j \in [k+1,n]. \]
Our procedure guarantees that $\mathcal{A}_{[k+1,n]}$ can be embedded into a quantum $L$-torus $\Gamma_{q,[k]}$, in which we seek for an element $D_{[k]}$ such that the map
\begin{equation}
\label{Deltak}
\Delta_{[k]}: \;\; \Gamma_{q,[k]} \longrightarrow \Gamma_{q,[k]}, \;\;
\Delta_{[k]}(a)=D_{[k]} \ast a-\sigma_{[k]}(a) \ast D_{[k]}, \;\; a \in \Gamma_{q,[k]},
\end{equation}
has the property that
\[
\Delta_{[k]}(\mathcal{A}_{[k+1,n]}) \subset \mathcal{A}_{[k+1,n]}.
\]
Once such an element $D_{[k]}$ is constructed, we define $\mathcal{A}_{[k,n]}$
to be the quantum-Ore extension $\mathcal{A}_{[k+1,n]}[X_k;\sigma_{[k]},\Delta_{[k]}]$ of $\mathcal{A}_{[k+1,n]}$. By induction, we
obtain an $L$-quantization $\mathcal{A}:=\mathcal{A}_{[1,n]}$ of $A$. This quantum $L$-algebra $\mathcal{A}$ has a presentation as a symmetric quantum-CGL extension
\[ \mathcal{A}=L[X_1;\sigma_1,\Delta_1]\cdots [X_n;\sigma_n,\Delta_n]_{(\lambda_1, \ldots, \lambda_n)}. \]
The quantum $L$-algebra $\mathcal{A}$ can also be viewed as the $L$-algebra generated by $n$ variables $X_1,...,X_n$ with the commutating relations:
\[
X_i \ast X_j=q^{-\lambda_{i,j}}X_j \ast X_i+\Delta_{i,j}, \ \ 1 \leq i < j \leq n,
\]
where $\Delta_{i,j}=\Delta_{[i]}(X_j)$.

\subsection{Notation}\label{subsec-not}

Let $\mathbf{k}$ be a field. Let $\mathbb{T}$ be a split $\mathbf{k}$-torus with the Lie algebra $\mathfrak{t}$. Denote by $\chi(\mathbb{T})$ and $\chi_{\ast}(\mathbb{T})$ the character lattice and the co-character lattice of $\mathbb{T}$ respectively. We identify both $\chi(\mathbb{T})$ and $\chi_{\ast}(\mathbb{T})$ with $\mathbb{Z}^r$ through the isomorphism $\mathbb{T} \cong (\mathbf{k}^{\times})^{r}$. We also identify the Lie algebra $\mathfrak{t}$ of $\mathbb{T}$ with $\chi(\mathbb{T}) \otimes_{\mathbb{Z}} \mathbf{k}$ and the dual space $\mathfrak{t}^*$ of $\mathfrak{t}$ with
$\chi_{\ast}(\mathbb{T}) \otimes_{\mathbb{Z}} \mathbf{k}$. An element $h \in \chi_{\ast}(\mathbb{T})$ can then be regarded both as a $\mathbf{k}$-group morphism $\mathbf{k}^\times \to \mathbb{T}$ and an element in $\mathfrak{t}$. Similarly, an element $\lambda \in \chi(\mathbb{T})$ can be regarded as a group homomorphism $\lambda: \mathbb{T} \to \mathbf{k}^\times, t \mapsto t^\lambda$ for $t \in \mathbb{T}$, as well as an element $\lambda \in \mathfrak{t}^*: x \mapsto \lambda(x)$ for $x \in \mathfrak{t}$.

We define a total order on $\mathbb{Z}^r$. For $(z_1,...,z_r)$ and $(z'_1,...,z'_r)$ in $\mathbb{Z}^r$, we say $(z_1,...,z_r) \leq (z'_1,...,z'_r)$ if one of the following conditions is satisfied:

1) $z_i=z'_i$, $i \in [1,r]$;

2) there exists $k \in [1,r]$ such that $z_k<z'_k$ and $z_j=z'_j, \ j \in [k+1,r].$\\
Through the identification between $\chi(\mathbb{T})$ and $\mathbb{Z}^r$, this total order induces a total order on $\chi(\mathbb{T})$, still denoted by $\leq$.

Let $A=\mathbf{k}[x_1,...,x_n]$ be a polynomial algebra over $\mathbf{k}$. As a $\mathbf{k}$-vector space, $A$ has a basis $\{x^{k_1}_1 \cdots x^{k_n}_n : (k_1,...,k_n) \in \mathbb{N}^n\}$. Define the degree of $x^{k_1}_1 \cdots x^{k_n}_n$ to be $f=(k_1,...,k_{n}) \in \mathbb{N}^n$. Equip $\mathbb{N}^n \subset \mathbb{Z}^n $ with the total order $\leq$ introduced above. For $a \in A$, define the degree of $a$ to be the degree of its highest degree term.

Let $S=R[X_1;\sigma_1,\Delta_1] \cdots [X_n;\sigma_n,\Delta_n]$ be an iterated Ore extension of a ring $R$. As an $R$-vector space,
$S$ has a basis $\{X^{k_1}_1 \cdots X^{k_n}_n : (k_1,...,k_n) \in \mathbb{N}^n\}$. Define the degree of $X^{k_1}_1 \cdots X^{k_n}_n$ to be $f=(k_1,...,k_{n}) \in \mathbb{N}^n$. Equip $\mathbb{N}^n \subset \mathbb{Z}^n $ with the total order $\leq$ introduced above. For $s \in S$, define the degree of $s$ to be the degree of its highest degree term.

\section{Poisson-CGL extensions}\label{sec-Poi-CGL}

Fix a field $\mathbf{k}$ of characteristic $0$ and let $\mathbb{T}$ be a split $\mathbf{k}$-torus. We follow the notation in $\S$\ref{subsec-not}.

\subsection{Preliminaries}

Let $A$ be a Poisson $\mathbf{k}$-algebra. Recall that $a \in A$ is said to be Poisson if the principal ideal $aA$ is a Poisson ideal.

\begin{definition}
{\rm
\label{ham-def}
For a Poisson element $a \in A$, the log-Hamiltonian derivation $H_{\log (a)}$ on $A$ is a derivation on $A$ given by
\[
H_{\log(a)}(b) = \frac{1}{a} \{a, \ b\}, \ \ b \in A.
\]
\hfill $\diamond$
}
\end{definition}

Recall (see \cite[$\S$1.4]{DM-p} and \cite[$\S$2]{BG}) that a $\mathbb{T}$-action on $A$ by $\mathbf{k}$-algebra automorphisms is said to be rational if $A$, as a $\mathbf{k}$-vector space, is generated by $\mathbb{T}$-weight vectors with weights in $\chi(\mathbb{T})$. We say $A$ is a $\mathbb{T}$-Poisson algebra if it is equipped with a rational $\mathbb{T}$-action by Poisson algebra automorphisms. An element $a \in A$ is said to be $\mathbb{T}$-homogeneous, or simply homogeneous, if it is a weight vector of the $\mathbb{T}$-action. The $\mathbb{T}$-action on $A$ induces an action of the Lie algebra $\mathfrak{t}$ on $A$ by Poisson derivations (see \cite[$\S$1.4]{DM-p} and \cite[$\S$2]{BG}). In such case, the Poisson derivation on $A$ corresponding to $h \in \mathfrak{t}$ is denoted by $\partial h$.

\begin{lemma}
\label{prime factor}
Let $A$ be a $\mathbb{T}$-Poisson algebra and a UFD. Let $a_1 \in A$ be a prime factor of $a$. Then

1) if $a$ is homogeneous, so is $a_1$;

2) if $a$ is Poisson, so is $a_1$.
\end{lemma}

\begin{proof}
1) Since the action of $\mathbb{T}$ on $A$ is rational, every nonzero element of $A$ is a finite sum of homogeneous elements with distinct weights. With the total order $\leq$ introduced in Section \ref{subsec-not}, for each nonzero $f \in A$, let $\lambda_{f, {\rm min}}$ and $\lambda_{f, {\rm max}}$ be the respective minimal and maximal weights appearing in the decomposition of $f$ into the sum of homogeneous elements. It is straightforward to check that the total order $\leq$ satisfies
\[
\lambda_1 \leq \lambda_2 \;\; \mbox{and} \;\; \lambda_3 \leq \lambda_4\;\; \Longrightarrow\;\;
\lambda_1 + \lambda_3 \leq \lambda_2 + \lambda_4
\]
for any $\lambda_j \in \chi(T)$, $j = 1, 2, 3, 4$. Write $a=a_1^ka_2$, where $k$ is a positive integer and $a_1$ does not divide $a_2$. It follows from
\[
\lambda_{a_1, {\rm min}} \lambda_{a^{k-1}_1a_2, {\rm min}} = \lambda_{a, {\rm min}} = \lambda_{a, {\rm max}} =
\lambda_{a_1, {\rm max}} \lambda_{a^{k-1}_1a_2, {\rm max}}
\]
that both $a_1$ and $a^{k-1}_1a_2$ are homogeneous.

2) Still write $a=a_1^ka_2$, where $k$ is a positive integer and $a_1$ does not divide $a_2$. For any $a' \in A$,
\[ \{a, \ a'\}=\{a_1^ka_2, \ a'\}=a^k_1\{a_2, \ a'\}+ka^{k-1}_1a_2\{a_1, \ a'\}. \]
Because $a$ is Poisson, $\{a, \ a'\} \in aA$. Hence, $\{a_1,a'\} \in a_1A$ for any $a' \in A$, which means $a_1$ is Poisson.
\end{proof}

\subsection{Poisson-CGL extensions}\label{subsec-Poi-CGL}

\begin{definition}
\cite[Definition 4.4]{PC}
\label{T-P-O}
{\rm
Given a Poisson $\mathbf{k}$-algebra $A$, a Poisson-Cauchon extension of $A$ is a Poisson-Ore extension $A[x;\theta,\delta]$ of $A$, together with a rational $\mathbb{T}$-action by Poisson algebra automorphisms, satisfying

1) $A \subset A[x]$ is $\mathbb{T}$-stable and $x \in A[x]$ is homogeneous with weight $\lambda_0 \in \chi(\mathbb{T})$;

2) $\delta$ is locally nilpotent;

3) there exists $h_0 \in \mathfrak{t}$ such that $\theta = \partial h_0|_{A}$ and $\eta:=\lambda_0(h_0) \neq 0$.\\
We denote this Poisson-Cauchon extension of $A$ by $A[x;\theta,\delta]_{\lambda_0}$.
\hfill $\diamond$
}
\end{definition}

\begin{remark}
\label{com-h}
{\rm
The $\mathbb{T}$-action on $A[x;\theta,\delta]$ is determined by the $\mathbb{T}$-action on $A$ and the weight $\lambda_0$ of $x$. On the other hand, the $\mathbb{T}$-action on $A[x;\theta,\delta]$ with Condition 1) is by Poisson algebra automorphisms if and only if
\[
t \cdot \{x, a\} = \{t\cdot x, \ t\cdot a\}, \ \ a \in A, \ t \in \mathbb{T},
\]
which is equivalent to
\begin{equation}\label{eq-ta}
t \cdot \delta(a) = t^{\lambda_0}\delta(t \cdot a), \ \ a \in A, \ t \in \mathbb{T}.
\end{equation}
As $\mathbb{T}$ is connected, \eqref{eq-ta} is, in turn, equivalent to
\begin{equation}\label{eq-ha}
[\partial h|_{A}, \ \delta] = \lambda_0(h)\delta, \ \ h \in \mathfrak{t}.
\end{equation}
In particular, one has
\[ [\theta, \ \delta]=\eta \delta. \]
\hfill $\diamond$
}
\end{remark}

Let $A$ be a Poisson $\mathbf{k}$-algebra and an integral domain. Suppose $A[x;\theta,\delta]_{\lambda_0}$ is a Poisson-Cauchon extension of $A$. Clearly, $A[x;\theta,\delta]_{\lambda_0}$ is still an integral domain. The following concept of "distinguished" is essential for the quantizations constructed in this paper.

\begin{definition}
\label{dis-def}
{\rm
Let $F$ be the fraction field of $A$ with the $\mathbb{T}$-action by $\mathbf{k}$-algebra automorphisms that extends the $\mathbb{T}$-action on $A$. An element $d \in F$ is said to be distinguished with respect to $A[x;\theta,\delta]_{\lambda_0}$ if it is homogeneous with weight $\lambda_0$ and satisfies
\begin{equation}\label{dis1}
\{d, \ r\}=\theta(r)d+\delta(r), \ \ r \in F.
\end{equation}
\hfill $\diamond$
}
\end{definition}

\begin{remark}
\label{ham-1}
{\rm
The log-Hamiltonian derivation $H_{\log (x-d)}$ of $x-d$ on $F$ (see Definition \ref{ham-def}) satisfies $H_{\log (x-d)}(A) \subset A$ and $H_{\log (x-d)}|_A=\theta$, i.e.,
\[ \{x-d, \ a \}=\theta(a)(x-d), \ \ a \in A. \]
\hfill $\diamond$
}
\end{remark}

\begin{lemma}
\label{dis-uni}
There is at most one distinguished element $d$ with respect to $A[x;\theta,\delta]_{\lambda_0}$ and it satisfies
\[ \delta(d)=-\eta d^2, \]
where $\eta$ is defined as in 2) of Definition \ref{T-P-O}.
\end{lemma}

\begin{proof}
Assume $d' \in F$ is also a distinguished element with respect to $A[x;\theta,\delta]_{\lambda_0}$. Let $e=d-d'$. Then $e$ is homogeneous with weight $\lambda_0$ and
\[ \{e, \ r\}=\theta(r)e, \ \ r \in F. \]
Since $0=\{e, \ e\}=\theta(e)e=\eta e^2$ and $\eta \neq 0$, $e=0$, which means $d=d'$. Therefore, there is at most one distinguished element $d \in F$ with respect to $A[x;\theta,\delta]_{\lambda_0}$.

It follows from
\[ 0=\{d, \ d\}=\theta(d)d+\delta(d)=\eta d^2+\delta(d) \]
that $\delta(d)=-\eta d^2.$
\end{proof}

We now turn to iterated Poisson-Cauchon extensions: starting from the ground field $\mathbf{k}$ (with the zero Poisson structure), one can form iterated Poisson-Ore extensions to get a Poisson $\mathbf{k}$-algebra
\[
A=\mathbf{k}[x_1;\theta_1,\delta_1]\cdots [x_n;\theta_n,\delta_n],
\]
which is called an iterated Poisson-Ore extension of $\mathbf{k}$, or simply an iterated Poisson-Ore extension. For $j \in [1,n]$, denote by $A_j$ the subalgebra of $A$ generated by $x_1,...,x_j$ and set $A_0 = \mathbf{k}$. The concept of iterated Poisson-Cauchon extensions is introduced by K. Goodearl and M. Yakimov in \cite[Definition 4]{QQ} and \cite[Definition 5.1]{PC}, which we now recall.

\begin{definition}
\label{T-Poi-CGL}
{\rm
A length $n$ Poisson-CGL extension is an iterated Poisson-Ore extension
\[ A=\mathbf{k}[x_1;\theta_1,\delta_1] \cdots [x_n;\theta_n,\delta_n], \]
together with a rational $\mathbb{T}$-action by Poisson algebra automorphisms, such that for $j \in [1, n]$,

1) $x_j$ is homogeneous with weight $\lambda_j \in \chi(\mathbb{T})$;

2) $\delta_j$ is locally nilpotent;

3) there exists $h_j \in \mathfrak{t}$ satisfying $\eta_j:=\lambda_{j}(h_{j}) \neq 0$ and
\[ \theta_{j}=\partial h_j|_{A_{j-1}}. \]
As the $\mathbb{T}$-action on $A$ is determined by $\lambda_1, \ldots, \lambda_n$, we also denote the Poisson algebra $A$ by
\[
A = \mathbf{k}[x_1;\theta_1,\delta_1]\cdots
[x_n;\theta_n,\delta_n]_{(\lambda_1, \ldots,
\lambda_n)}.
\]
\hfill $\diamond$
}
\end{definition}

A Poisson-CGL extension $A$ is thus a sequence
\[
\mathbf{k} = A_0 \subset A_1 \subset \cdots \subset A_n = A
\]
of $\mathbb{T}$-Poisson algebras such that $A_j = A_{j-1}[x_j;\theta_j,\delta_j]_{\lambda_j}$ is a Poisson-Cauchon extension of $A_{j-1}$ for $j \in [1,n]$. It is not difficult to see that, for $j \in [1,n]$, $A_j$ is a Poisson-CGL extension.

One of the most striking features of a Poisson-CGL extension $A$ concerns the homogeneous Poisson prime elements of the nested
sequence of Poisson-CGL extensions
\[
A_1 \subset \cdots \subset A_n = A.
\]
We now state, in Theorem \ref{HPP-Poi-CGL}, some results from \cite{PC}, which are relevant to this paper. For a Poisson-CGL extension
\[ A = \mathbf{k}[x_1;\theta_1,\delta_1]\cdots [x_n;\theta_n,\delta_n]_{(\lambda_1, \ldots, \lambda_n)}, \]
and for $j \in [1,n]$, let $P_j \subset A_j$ be the set of all homogeneous Poisson prime elements of $A_j$, and let $Q_j \subset P_j$ be the set of elements in $P_j$ that are not in $A_{j-1}$. Clearly both $P_j$ and $Q_j$ are invariant under multiplication by scalars in
$\mathbf{k}^\times$. Define $P_0 = Q_0 = \mathbf{k}^\times$.

\begin{theorem}
\label{HPP-Poi-CGL}
Let $A = \mathbf{k}[x_1;\theta_1,\delta_1] \cdots [x_n;\theta_n,\delta_n]_{(\lambda_1, \ldots,
\lambda_n)}$ be a Poisson-CGL extension. For $j \in [1,n]$, let $\eta_j$ be as in 3) of Definition \ref{T-Poi-CGL}. Then for each $j \in [1, n]$, $|Q_j/\mathbf{k}^\times|=1$, and
\[
P_j =Q_j \sqcup \bigsqcup_{i \in I_j} Q_i,
\]
where $I_j =  \{1 \leq i \leq j-1: \ \delta_{i+1}|_{Q_i} = \cdots = \delta_j|_{Q_i} = 0\}$. Moreover, there is a unique sequence $y_A = \{y_1,...,y_n\}$ with $y_j \in Q_j$ for $j \in [1, n]$, which is determined inductively as follows: $y_0 = 1 \in Q_0$, and for $j \geq 1$,
\[
y_j = y_{p(j)}x_j - c_j = y_{p(j)}x_j - \frac{\delta_j(y_{p(j)})}{\eta_j}=y_{p(j)}(x_j-d_j),
\]
where $d_j$ is the distinguished element with respect to $A_{j-1}[x_j;\theta_j,\delta_j]_{\lambda_j}$. Here, $p(j)=0$ when $\delta_j=0$, and when $\delta_j \neq 0$, $p(j)$ is an (necessarily unique) integer in $[1, j-1]$ such that $y_{p(j)} \in P_{j-1}$ and $\delta_j(y_{p(j)}) \neq 0$.
\end{theorem}

Note that all the elements in $y_A$ are homogeneous and prime in $A$, but although each $y_j \in y_A$ is a Poisson element of $A_j$, only $y_n$ and those $y_j$ with $j \in I_n$ are Poisson elements of $A$.

\begin{definition}
\label{y_A}
{\rm We refer to the sequence $y_A = \{y_1, \ldots, y_n\}$ in Theorem \ref{HPP-Poi-CGL} as the sequence of homogeneous Poisson prime elements of $A$.
\hfill $\diamond$
}
\end{definition}

Let $A = \mathbf{k}[x_1;\theta_1,\delta_1] \cdots [x_n;\theta_n,\delta_n]_{(\lambda_1, \ldots, \lambda_n)}$ be a Poisson-CGL extension and let $y_A = \{y_1, \ldots, y_n\}$ be the sequence of homogeneous Poisson prime elements of $A$. Note that the map $p$, defined as in Theorem \ref{HPP-Poi-CGL} for $A$, is injective when restricted to $\{ j \in [1, n]: \delta_j \neq 0\}$.

\begin{definition}
\label{level}
{\rm
Two integers $i < j$ in $[1, n]$ are said to have the same level if $i=p^m(j)$ for some positive integer $m$; The set of all
integers in the same level is called a level set of $A$; The number of level sets is called the rank of $A$. (Note that the level sets of $A$ are exactly the level sets of the function $\eta$ for $A$ in \cite[Theorem 5.5]{PC}).
\hfill $\diamond$
}
\end{definition}

The following lemma follows directly from Theorem \ref{HPP-Poi-CGL}.

\begin{lemma}
\label{one-one}
1) For each level set $K$, let $k_{max}$ be the largest integer in $K$. Then $y_{k_{max}}$ is a homogeneous Poisson prime element of $A$. Conversely, any homogeneous Poisson prime element of $A$ is a scalar multiple of $y_{k_{max}}$ for some level set $K$.

2) Define the $\mathbb{N}^n$-valued degrees for elements in $A$ as in $\S$\ref{subsec-not}. For $j \in [1,n]$, let $f_j=(f_{j,1},...,f_{j,n}) \in \mathbb{N}^n$ be the degree of $y_{j}$ and let $K_j$ be the level set that contains $j$. For $j,l \in [1,n]$, $f_{j,l}=1$ if $l \in [1,j] \cap K_j$ and $f_{j,l}=0$ otherwise.
\end{lemma}

By Theorem \ref{HPP-Poi-CGL}, for $j \in [1,n]$, $y_j = y_{p(j)}x_j - c_j$, where $c_j \in A_{j-1}$. It follows that the algebra $A$ is contained in the Laurent polynomial ring generated by $y_A$, i.e.,
\[ A \subset \Gamma:=\mathbf{k}[y^{\pm 1}_1,...,y^{\pm 1}_n]. \]
Equip $\Gamma$ with the extended Poisson structure from $A$. As a $\mathbf{k}$-vector space, $\Gamma$ has a natural basis
\[ \{y^{k_1}_1 \cdots y^{k_n}_n : (k_1,...,k_n) \in \mathbb{Z}^n\}. \]
For $v=(v_1,...,v_n) \in \mathbb{Z}^n$, define $y^v=\prod^n_{i=1}y^{v_i}_i$. The Poisson structure on $\Gamma$ defines a skew-symmetric bilinear form
\[ \Lambda_Y : \mathbb{Z}^n \otimes \mathbb{Z}^n \rightarrow \mathbb{Z} \]
such that
\[ \{y^v, \ y^w\}=\Lambda_Y(v,w)y^{v+w}, \ \ v,w \in \mathbb{Z}^n. \]

\begin{definition}
\label{Poi-tor}
{\rm
The algebra $\Gamma$, equipped with the extended Poisson structure from $A$, is called the Poisson torus of $A$; The matrix $\Lambda_Y:=(\kappa_{i,j})_{i,j \in [1,n]}$ is called the Poisson matrix of $\Gamma$.
\hfill $\diamond$
}
\end{definition}

The following lemma-notation is crucial for the construction of quantization in $\S$\ref{sec-qua}.

\begin{lemma-notation}
\label{mon}
For $k \in [1,n]$ such that $p(k) \neq 0$ (or equivalently $\delta_k \neq 0$),
\[ \Gamma_{k-1}:=\mathbf{k}[y^{\pm 1}_1,...,y^{\pm 1}_{k-1}] \ \ \text{and} \ \ \Gamma'_{k-1}:=\mathbf{k}[y^{\pm 1}_1,...,y^{\pm 1}_{p(k)-1},y^{\pm 1}_{p(k)+1},...,y^{\pm 1}_{k-1}]. \]
Let $\Lambda_{Y,k-1}=(\kappa_{i,j})_{i,j \in [1,k-1]}$ be the Poisson matrix of $\Gamma_{k-1}$. Then $c_k \in \Gamma'_{k-1}[y_{p(k)}]$, the constant term of which, denoted by $b_k$, is a nonzero monomial in $\Gamma'_{k-1}$. Moreover, $b_k$ is a scalar multiple of $y^v \in \Gamma_{k-1}$, where $v$ is the unique vector in $\mathbb{Z}^{k-1}$ satisfying

1) $y^v$ has the same weight as $y_k$;

2) $\Lambda_{Y,k-1}(v,e_l)=\left\{
\begin{array}{ll}
\kappa_{k,l} & l \in [1,k-1] \setminus \{p(k)\}\\
\kappa_{k,p(k)}+\eta_k & l=p(k)
\end{array}\right.$
\end{lemma-notation}

\begin{proof}
For $j \in [1,n]$, let $\eta_j$ be as in 3) of Definition \ref{T-Poi-CGL}. By Theorem \ref{HPP-Poi-CGL}, for $l \in [1,k-1]$,
\[ x_l=\frac{y_l+c_l}{y_{p(l)}}, \ \text{where} \ c_l \in A_{l-1}. \]
Recall that the map $p$, when restricted to $\{ j \in [1, n]: \delta_j \neq 0\}$, is injective. Since $\delta_k \neq 0$, $p(k) \neq p(l)$ for any $l \in [1,k-1]$. As $A_{k-1}$ is generated by those $x_1,...,x_{k-1}$, $A_{k-1} \subset \Gamma'_{k-1}[y_{p(k)}]$. Because $c_k \in A_{k-1}$, $c_k \in \Gamma'_{k-1}[y_{p(k)}]$.

Since $\delta_k \neq 0$, by Theorem \ref{HPP-Poi-CGL}, $c_k \neq 0$. Assume $b_k=0$. Then $c_k=y_{p(k)}\frac{a_1}{a_2}$, where $a_1,a_2 \in A_{k-1}$ and $a_2 \notin y_{p(k)}A_{k-1}$. One sees that $c_ka_2 \in y_{p(k)}A_{k-1}$. As $y_{p(k)}$ is prime in $A_{k-1}$ and $a_2 \notin y_{p(k)}A_{k-1}$, one has $c_k \in y_{p(k)}A_{k-1}$, which contradicts to $y_k=y_{p(k)}x_k-c_k$ being prime in $A_k$. Thus, $b_k \neq 0$.

By Theorem \ref{HPP-Poi-CGL}, the distinguished element with respect to $A_{k-1}[x_k;\theta_k,\delta_k]_{\lambda_k}$ is
\[ d_k=\frac{c_k}{y_{p(k)}}. \]
Since $b_k$ is the constant term of $c_k \in \Gamma'_{k-1}[y_{p(k)}]$, $\frac{b_k}{y_{p(k)}}$ is a monomial term of $d_k \in \Gamma'_{k-1}[y^{\pm 1}_{p(k)}]$ and
\[ d_k-\frac{b_k}{y_{p(k)}} \in \Gamma'_{k-1}[y_{p(k)}]. \]
For $l \in [1,k-1]$, by (\ref{dis1}),
\[ \{d_k, \ y_l\}-\theta_k(y_l)d_k=\delta_k(y_l) \in A_{k-1} \subset \Gamma'_{k-1}[y_{p(k)}]. \]
It is easy to see that $\Gamma'_{k-1}[y_{p(k)}]$ is a Poisson subalgebra of $\Gamma_{k-1}$, so for $l \in [1,k-1]$,
\[ \{d_k-\frac{b_k}{y_{p(k)}}, \ y_l\}-\theta_k(y_l)(d_k-\frac{b_k}{y_{p(k)}}) \in \Gamma'_{k-1}[y_{p(k)}]. \]
Hence, for $l \in [1,k-1]$,
\[ \{\frac{b_k}{y_{p(k)}}, \ y_l\}-\theta_k(y_l)\frac{b_k}{y_{p(k)}} \in \Gamma'_{k-1}[y_{p(k)}]. \]
Therefore,
\begin{equation}
\label{equ-mon-1}
\{\frac{b_k}{y_{p(k)}}, \ y_l\}=\theta_k(y_l)\frac{b_k}{y_{p(k)}}, \ \ l \in [1,k-1] \setminus \{p(k)\}.
\end{equation}
By Theorem \ref{HPP-Poi-CGL}, $\delta_k(y_{p(k)})=\eta_k y_{p(k)}d_k$. By (\ref{dis1}),
\[ \{d_k, \ y_{p(k)}\}=(\theta_{k}(y_{p(k)})+\eta_k y_{p(k)})d_k. \]
As $\frac{b_k}{y_{p(k)}}$ is a monomial term of $d_k \in \Gamma'_{k-1}[y^{\pm 1}_{p(k)}]$,
\begin{equation}
\label{equ-mon-2}
\{\frac{b_k}{y_{p(k)}}, \ y_{p(k)}\}=(\theta_{k}(y_{p(k)})+\eta_ky_{p(k)})\frac{b_k}{y_{p(k)}}.
\end{equation}
Through direct computation and using (\ref{equ-mon-1}) (\ref{equ-mon-2}), one has
\begin{equation}
\label{equ-mon-3}
\{b_k, \ y_{l}\}=\left\{
\begin{array}{ll}
\kappa_{k,l}y_lb_k & l \in [1,k-1] \setminus \{p(k)\}\\
(\kappa_{k,p(k)}+\eta_k)y_{p(k)}b_k & l=p(k)
\end{array}\right.
\end{equation}
It is straightforward to see that every monomial term $b'_k$ of $b_k$ has the same weight as $y_k$ and satisfies (\ref{equ-mon-3}) by replacing $b_k$. As $b'_k$ is a monomial in $\Gamma'_{k-1}$, it is a scalar multiple of $y^v \in \Gamma_{k-1}$ for some $v \in \mathbb{Z}^{k-1}$. Then $y^v$ has the same weight as $y_k$ and
\[ \Lambda_{Y,k-1}(v,e_l)=\left\{
\begin{array}{ll}
\kappa_{k,l} & l \in [1,k-1] \setminus \{p(k)\}\\
\kappa_{k,p(k)}+\eta_k & l=p(k)
\end{array}\right. \]
By \cite[Lemma 11.10]{PC}, such $v$ is unique and $b_k$ has to be a monomial in $\Gamma'_{k-1}$.
\end{proof}

\subsection{Symmetric Poisson-CGL extensions}\label{subsec-sym-Poi-CGL}

Most interesting examples of Poisson-CGL extensions satisfy the "symmetry" condition:

\begin{definition}
\cite[Definition 4]{QQ}\cite[Definition 6.1]{PC}
\label{sym}
{\rm
Let
\[ A=\mathbf{k}[x_1;\theta_1,\delta_1] \cdots [x_n;\theta_n,\delta_n]_{(\lambda_1,...,\lambda_n)} \]
be a Poisson-CGL extension. Then $A$ is said to be symmetric if

1) for $1 \leq i <j \leq n$, $\delta_j(x_i) \in A_{[i+1,j-1]}$;

2) there exist $h'_1,...,h'_n \in \mathfrak{t}$ such that

\indent \indent a) $\theta_j(x_i)=-\lambda_j(h'_i)x_i$ for $1 \leq i < j \leq n$,

\indent \indent b) $\lambda_{j}(h'_{j}) \neq 0$ for $j \in [1,n]$.\\
Here, $A_{[i+1,j-1]}$ is the subalgebra of $A$ generated by $x_{i+1},...,x_{j-1}$.
\hfill $\diamond$
}
\end{definition}

Let $B=\mathbf{k}[x_0;\theta_0,\delta_0][x_1;\theta_1,\delta_1] \cdots [x_n;\theta_n,\delta_n]_{(\lambda_0,\lambda_1,...,\lambda_n)}$ be a symmetric Poisson-CGL extension. Then the subalgebra
\[ A:=\mathbf{k}[x_1;\theta_1,\delta_1] \cdots [x_n;\theta_n,\delta_n]_{(\lambda_1,...,\lambda_n)}\]
is also a symmetric Poisson-CGL extension. Let $p$ be the map for $A$ as in Theorem \ref{HPP-Poi-CGL}. Let $y_A=(y_1,...,y_n)$ be the sequence of homogeneous Poisson prime elements of $A$ and let $\Gamma$ be the Poisson torus of $A$.

Note that $B$ can be viewed as a Poisson-Cauchon extension $B=A[x_0;\theta,\delta]_{\lambda_0}$ of $A$, where
\[ \theta(x_j)x_0=-\theta_j(x_0)x_j \ \ \text{and} \ \ \delta(x_j)=-\delta_j(x_0), \ \ j \in [1,n].\]
Let $h_0,h_1,...,h_n$ be in $\mathfrak{t}$ such that they satisfy 3) of Definition \ref{T-Poi-CGL} for $B$. Let $h'_0,h'_1,...,h'_n$ be in $\mathfrak{t}$ such that they satisfy 2) of Definition \ref{sym} for $B$. Then for $j \in [1,n]$, $A_{j}[x_0;\theta,\delta]_{\lambda_0}$ (here $\theta$ and $\delta$ are their restrictions on $A_j$) is a Poisson-Cauchon extension of $A_j$, for which $h'_0$ satisfies 2) of Definition \ref{T-P-O}. Furthermore,
\[ \mathbf{k}[x_1;\theta_1,\delta_1] \cdots [x_j;\theta_j,\delta_j][x_0;\theta,\delta]_{(\lambda_1,...,\lambda_j,\lambda_0)} \]
is a Poisson-CGL extension, for which $h_1,...,h_j,h'_0$ satisfy 3) of Definition \ref{T-Poi-CGL}.

\begin{lemma}
\label{k-p(k)}
For $k \in [1,n]$ such that $p(k) \neq 0$, $\delta(y_k)=0$ if and only if $\delta(y_{p(k)})=0$.
\end{lemma}

\begin{proof}
1) By Theorem \ref{HPP-Poi-CGL}, $y_k=y_{p(k)}x_k-c_k$, where $c_k \in A_{k-1}$. Then
\[ \delta(y_k)=\delta(y_{p(k)}x_k-c_k)=\delta(y_{p(k)})x_k+y_{p(k)}\delta(x_k)-\delta(c_k). \]
Since $B$ is symmetric, $y_{p(k)}\delta(x_k)-\delta(c_k) \in A_{k-1}$. Then $\delta(y_k)=0$ would imply $\delta(y_{p(k)})=0$.

2) Conversely, assume $\delta(y_{p(k)})=0$ but $\delta(y_k) \neq 0$. Recall that $P_{k}$ and $P_{k-1}$ are the sets of all homogeneous Poisson prime elements of $A_k$ and $A_{k-1}$ respectively. Applying Theorem \ref{HPP-Poi-CGL} to
\[ \mathbf{k}[x_1;\theta_1,\delta_1] \cdots [x_k;\theta_k,\delta_k][x_0;\theta,\delta]_{(\lambda_1,...,\lambda_k,\lambda_0)}, \]
one sees that $\delta$ vanishes on $P_k \setminus (\mathbf{k}^{\times}y_k)$. Still by Theorem \ref{HPP-Poi-CGL},
\[ P_{k-1}=(P_k \setminus (\mathbf{k}^{\times}y_k)) \cup (\mathbf{k}^{\times}y_{p(k)}). \]
Since $\delta(y_{p(k)})=0$ by assumption, $\delta(P_{k-1})=0$. Applying Theorem \ref{HPP-Poi-CGL} to
\[ \mathbf{k}[x_1;\theta_1,\delta_1] \cdots [x_{k-1};\theta_{k-1},\delta_{k-1}][x_0;\theta,\delta]_{(\lambda_1,...,\lambda_{k-1},\lambda_0)},\]
one has $\delta=0$ on $A_{k-1}$ and $x_0,y_{p(k)}$ are both homogeneous Poisson prime elements of $A_{k-1}[x_0]$. Applying Theorem \ref{HPP-Poi-CGL} to
\[ \mathbf{k}[x_0;\theta_0,\delta_0][x_1;\theta_1,\delta_1] \cdots [x_{k};\theta_{k},\delta_{k}]_{(\lambda_0,\lambda_1,...,\lambda_{k})},\]
one has $\delta_k(x_0)=0$ since $\delta_k(y_{p(k)}) \neq 0$. It follows from $\delta(x_k)=-\delta_k(x_0)$ that $\delta(x_k)=0$. Then $\delta=0$ on $A_k$, contradicting to $\delta(y_k) \neq 0$. Therefore, $\delta(y_{p(k)})=0$ would imply $\delta(y_{k})=0$.
\end{proof}

In the rest of this section, assume $\delta \neq 0$. By Theorem \ref{HPP-Poi-CGL}, there exists a unique $k \in [1,n]$ such that $\delta(y_k) \neq 0$ and $y_k$ is a homogeneous Poisson prime element of $A$. Let $\eta$ be as in 2) of Definition \ref{T-P-O} for $A[x_0;\theta,\delta]_{\lambda_0}$. Let $m$ be the largest nonnegative integer such that $p^{m}(k) \in [1,n]$ and define
\[ d^{(m+1)}=0 \ \ \text{and} \ \ d^{(i)}=\frac{\delta(y_{p^i(k)})}{\eta y_{p^i(k)}} \in \Gamma. \]
By Theorem \ref{HPP-Poi-CGL}, $d^{(i)}$ is the distinguished element with respect to $A_{p^{i-1}(k)-1}[x_0;\theta,\delta]_{\lambda_0}$ ($p^{-1}(k):=n+1$), for $i \in [0,m]$.

\begin{lemma}
\label{d-M}
1) Let $i \in [0,m]$ and one has
\[ d^{(i)}-d^{(i+1)}=\frac{M^{(i)}}{y_{p^i(k)}y_{p^{i+1}(k)}},\]
where $M^{(i)}$ is either a nonzero constant or a product of homogeneous Poisson prime elements of $A_{p^{i}(k)-1}$.\\
2) Let $i \in [0,m-1]$ and let $b_{p^{i}(k)}$ be as in Lemma-Notation \ref{mon}. One has
\[ d^{(i)}-d^{(i+1)}=\frac{b_{p^{i}(k)}}{y_{p^{i}(k)}}(d^{(i+1)}-d^{(i+2)}). \]
\end{lemma}

\begin{proof}
1) Let $F$ be the fraction field of $A$. Every nonzero element in $F$ is Poisson and has its log-Hamiltonian derivations on $F$ (see Definition \ref{ham-def}). By (\ref{dis1}) and by the definitions of $d^{(i)}$ and $d^{(i+1)}$, for $a \in A_{p^{i}(k)-1}$,
\[ \{d^{(i)}, \ a\}=\theta(a)d^{(i)}+\delta(a) \ \ \text{and} \ \ \{d^{(i+1)}, \ a\}=\theta(a)d^{(i+1)}+\delta(a). \]
Define $g_i=d^{(i)}-d^{(i+1)}$ and one has $H_{\log (g_i)}|_{A_{p^{i}(k)-1}}=\theta|_{A_{p^{i}(k)-1}}$, which clearly implies $H_{\log (g_i)}(A_{p^{i}(k)-1}) \subset A_{p^{i}(k)-1}$. By Theorem \ref{HPP-Poi-CGL}, $y_{p^{i}(k)}=y_{p^{i+1}(k)}(x_{p^{i}(k)}-d_{p^{i}(k)})$. By Remark \ref{ham-1},
\begin{align*}
H_{\log (y_{p^{i}(k)})}(A_{p^{i}(k)-1})&=H_{\log (x_{p^{i}(k)}-d_{p^{i}(k)})}(A_{p^{i}(k)-1})+H_{\log (y_{p^{i+1}(k)})}(A_{p^{i}(k)-1}) \\ &=\theta_{p^{i}(k)}(A_{p^{i}(k)-1})+H_{\log (y_{p^{i+1}(k)})}(A_{p^{i}(k)-1}).
\end{align*}
Since $y_{p^{i+1}(k)}$ is Poisson in $A_{p^{i}(k)-1}$, $H_{\log (y_{p^{i+1}(k)})}(A_{p^{i}(k)-1}) \subset A_{p^{i}(k)-1}$. Thus,
\[ H_{\log (y_{p^{i}(k)})}(A_{p^{i}(k)-1}) \subset A_{p^{i}(k)-1}. \]
Define $M^{(i)}=y_{p^{i+1}(k)}y_{p^{i}(k)}g_i$. It follows from
\[
H_{\log (M^{(i)})}=H_{\log (y_{p^{i}(k)})}+H_{\log (y_{p^{i+1}(k)})}+ H_{\log (g_i)}
\]
that $H_{\log (M^{(i)})}(A_{p^{i}(k)-1}) \subset A_{p^{i}(k)-1}$. Through direct computation, one sees that $M^{(i)} \in A_{p^{i}(k)-1}$ since $B$ is symmetric. It follows that $M^{(i)}$ is a homogeneous Poisson element of $A_{p^{i}(k)-1}$. By Lemma \ref{prime factor}, $M^{(i)}$ is either a constant or a product of homogeneous Poisson prime elements of $A_{p^{i}(k)-1}$. Suppose $M^{(i)}=0$. Then $d^{(i)}-d^{(i+1)}=0$, which means $y_{p^i(k)}\delta(y_{p^{i+1}(k)})=y_{p^{i+1}(k)}\delta(y_{p^i(k)})$. Since $y_{p^i(k)}$ is prime in $A_{p^{i}(k)}$ and clearly $y_{p^{i+1}(k)} \notin y_{p^i(k)}A_{p^{i}(k)}$, $\delta(y_{p^i(k)}) \in y_{p^i(k)}A_{p^{i}(k)}$. Applying Theorem \ref{HPP-Poi-CGL} to
\[ \mathbf{k}[x_1;\theta_1,\delta_1] \cdots [x_{p^i(k)};\theta_{p^i(k)},\delta_{p^i(k)}][x_0;\theta,\delta]_{(\lambda_1,...,\lambda_{p^i(k)},\lambda_0)}, \]
as $\delta(y_{p^i(k)}) \neq 0$ by Lemma \ref{k-p(k)}, one sees that $y_{p^i(k)}x_0-\frac{\delta(y_{p^i(k)})}{\eta}$ is prime in $A_{p^i(k)}[x_0]$, contradicting to $\delta(y_{p^i(k)}) \in y_{p^i(k)}A_{p^{i}(k)}$. Thus, $M^{(i)}$ is nonzero. As $g_i=d^{(i)}-d^{(i+1)}$ and $M^{(i)}=y_{p^{i+1}(k)}y_{p^{i}(k)}g_i$,
\[ d^{(i)}-d^{(i+1)}=\frac{M^{(i)}}{y_{p^i(k)}y_{p^{i+1}(k)}}.\]

2) As $y_{p^{i}(k)}=y_{p^{i+1}(k)}x_{p^{i}(k)}-c_{p^{i}(k)}$, one has
\begin{align*}
M^{(i)}&=y_{p^{i+1}(k)}y_{p^{i}(k)}(d^{(i)}-d^{(i+1)})\\
&=y_{p^{i+1}(k)}y_{p^{i}(k)}d^{(i)}-y^2_{p^{i+1}(k)}x_{p^{i}(k)}d^{(i+1)}+y_{p^{i+1}(k)}c_{p^{i}(k)}d^{(i+1)}.
\end{align*}
It follows from
\[ y_{p^{i+1}(k)}y_{p^{i}(k)}d^{(i)}-y^2_{p^{i+1}(k)}x_{p^{i}(k)}d^{(i+1)} \in y_{p^{i+1}(k)}A_{p^{i}(k)-1} \]
that
\begin{equation}
\label{degree}
M^{(i)}-y_{p^{i+1}(k)}c_{p^{i}(k)}d^{(i+1)} \in y_{p^{i+1}(k)}A_{p^{i}(k)-1}.
\end{equation}
Recall from Lemma-Notation \ref{mon} that $\Gamma_{p^{i}(k)-1}=\mathbf{k}[y^{\pm 1}_1,...,y^{\pm 1}_{p^{i}(k)-1}]$ and $\Gamma'_{p^{i}(k)-1}=\mathbf{k}[y^{\pm 1}_1,...,y^{\pm 1}_{p^{i+1}(k)-1},y^{\pm 1}_{p^{i+1}(k)+1},...,y^{\pm 1}_{p^{i}(k)-1}]$. It is easy to see that
\[ \Gamma_{p^{i}(k)-1}=\Gamma'_{p^{i}(k)-1}[y^{\pm 1}_{p^{i+1}(k)}]. \]
Now consider every element in $\Gamma_{p^{i}(k)-1}$, including those involved in (\ref{degree}), as an element in $\Gamma'_{p^{i}(k)-1}[y^{\pm 1}_{p^{i+1}(k)}]$. The lowest degree term of $c_{p^{i}(k)}$ is $b_{p^{i}(k)}$ with degree 0. By part 1), the lowest degree term of $d^{(i+1)}$ is $d^{(i+1)}-d^{(i+2)}$ with degree -1. Thus, the lowest degree term of $y_{p^{i+1}(k)}c_{p^{i}(k)}d^{(i+1)}$ is $y_{p^{i+1}(k)}b_{p^{i}(k)}(d^{(i+1)}-d^{(i+2)})$ with degree 0. By (\ref{degree}),
the lowest degree term of $M^{(i)}-y_{p^{i+1}(k)}c_{p^{i}(k)}d^{(i+1)}$ has degree greater than 0. By 1) of Lemma \ref{one-one}, $M^{(i)}$ is a monomial in $\Gamma_{p^{i}(k)-1}$ because $M^{(i)}$ is either a nonzero constant or a product of homogeneous Poisson prime elements of $A_{p^{i}(k)-1}$. Hence,
\[ M^{(i)}=y_{p^{i+1}(k)}b_{p^{i}(k)}(d^{(i+1)}-d^{(i+2)}). \]
Therefore, by part 1),
\[ d^{(i)}-d^{(i+1)}=\frac{b_{p^{i}(k)}(d^{(i+1)}-d^{(i+2)})}{y_{p^{i}(k)}}. \]
\end{proof}

\section{CGL Extensions}\label{sec-CGL}

Fix a field $\mathbf{k}$ of any characteristic and let $\mathbb{T}$ be a split $\mathbf{k}$-torus. We follow the notation in $\S$\ref{subsec-not}.

\subsection{CGL extensions}\label{subsec-CGL}

For every Ore extension $S[X;\sigma,\Delta]$ appearing in $\S$\ref{sec-CGL}, $S$ is an associative $\mathbf{k}$-algebra with unity, and $\sigma,\Delta$ are $\mathbf{k}$-linear maps. As a result, $S[X;\sigma,\Delta]$ is an associative $\mathbf{k}$-algebra with unity. Recall that an action of $\mathbb{T}$ on $S$ by $\mathbf{k}$-algebra automorphisms is said to be rational if $S$, as a $\mathbf{k}$-vector space, is generated by $\mathbb{T}$-weight vectors. Suppose there is a rational $\mathbb{T}$-action on $S$ by $\mathbf{k}$-algebra automorphisms. An element $s \in S$ is said to be $\mathbb{T}$-homogeneous, or simply homogeneous, if it is a weight vector of the $\mathbb{T}$-action.

CGL extensions can be viewed as iterated Cauchon extensions, which are introduced in \cite[Definition 2.5]{LLR}. To serve the purpose of this paper, we use a slightly different definition of Cauchon extensions.

\begin{definition}
\label{T-O}
{\rm
Given an associative $\mathbf{k}$-algebra $S$ with unity, a Cauchon extension of $S$ is an Ore extension $S[X;\sigma,\Delta]$, together with a rational $\mathbb{T}$-action by $\mathbf{k}$-algebra automorphisms, satisfying

1) $S \subset S[X]$ is $\mathbb{T}$-stable and $X$ is homogeneous with weight $\lambda_0 \in \chi(\mathbb{T})$;

2) $\Delta$ is locally nilpotent;

3) there exists $t_0 \in \mathbb{T}$ such that $\sigma = t_0 \cdot |_{S}$ and $\omega:=t^{\lambda_0}_0 \in \mathbf{k}$ is not a root of unity.\\
We denote this Cauchon extension of $S$ by $S[X;\sigma,\Delta]_{\lambda_0}$.
\hfill $\diamond$
}
\end{definition}

\begin{remark}
{\rm
\label{com-t}
The $\mathbb{T}$-action on $S[X;\sigma,\Delta]$ is determined by the $\mathbb{T}$-action on $S$ and the weight $\lambda_0$ of $x$. Since the $\mathbb{T}$-action on $S$ is by $\mathbf{k}$-algebra automorphisms,
\[ t \cdot (\Delta(s))=t^{\lambda_0} \Delta(t \cdot s), \ \ s \in S, \ t \in \mathbb{T}. \]
In particular, one has
\[ \sigma \circ \Delta=\omega \Delta \circ \sigma. \]
\hfill $\diamond$
}
\end{remark}

Let $S$ be an associative $\mathbf{k}$-algebra with unity and a right Ore domain. Let $S[X;\sigma,\Delta]_{\lambda_0}$ be a Cauchon extension of $S$. By \cite[Proposition 3.4]{Cohn}, $S[X;\sigma,\Delta]_{\lambda_0}$ is still a right Ore domain. Let $\mathcal{F}$ be the division ring of fractions of $S$ with the $\mathbb{T}$-action by $\mathbf{k}$-algebra automorphisms that extends the $\mathbb{T}$-action on $S$.

\begin{definition}
\label{qdis-def}
{\rm
An element $D \in \mathcal{F}$ is said to be distinguished with respect to $S[X;\sigma,\Delta]_{\lambda_0}$ if it is homogeneous with weight $\lambda_0$ and satisfies
\begin{equation}\label{qdis1}
Dr=\sigma(r)D+\Delta(r), \ \ r \in \mathcal{F}.
\end{equation}
\hfill $\diamond$
}
\end{definition}

\begin{lemma}
\label{qdis-uni}
There is at most one distinguished element $D$ with respect to $S[X;\sigma,\Delta]_{\lambda_0}$ and it satisfies
\[ \Delta(D)=(1-\omega)D^2, \]
where $\omega$ is defined as in 2) of Definition \ref{T-O}.
\end{lemma}

\begin{proof}
Assume $D' \in \mathcal{F}$ is also a distinguished element with respect to $S[D;\sigma,\Delta]_{\lambda_0}$. Then $D'$ is homogeneous with weight $\lambda_0$ and
\[ Dr=\sigma(r)D+\Delta(r), \ \ r \in \mathcal{F}. \]
Let $e=D-D'$. Then $e$ is homogeneous with weight $\lambda_0$ and $er=\sigma(r)e$ for $r \in \mathcal{F}$. Since $e^2=\sigma(e)e=\omega e^2$ and $\omega \neq 1$, one sees that $e=0$. Thus, $D=D'$. Therefore, there is at most one distinguished element $D \in \mathcal{F}$ with respect to $S[D;\sigma,\Delta]_{\lambda_0}$. Moreover, as
\[ D^2=\sigma(D)D+\Delta(D)=\omega D^2+\Delta(D), \]
one has $\Delta(D)=(1-\omega)D^2.$
\end{proof}

We now turn to iterated Cauchon extensions. Given an iterated Ore extension
\[ S=\mathbf{k}[X_1,\sigma_1,\Delta_1]\cdots [X_n;\sigma_n,\Delta_n], \]
denote by $S_j$ the subalgebra of $S$ generated by $X_1,...,X_{j}$ and set $S_0 = \mathbf{k}$.

\begin{definition}
\cite[Definition 1]{QQ} \cite[Definition 4.1]{QO} \cite[Definition 3.3]{QQ2}
\label{T-CGL}
{\rm
A length $n$ CGL extension is an iterated Ore extension $S=\mathbf{k}[X_1,\sigma_1,\Delta_1]\cdots [X_n;\sigma_n,\Delta_n]$, together with a rational $\mathbb{T}$-action by $\mathbf{k}$-algebra automorphisms, such that for each $j \in [1, n]$,

1) $X_j$ is homogeneous with weight $\lambda_j \in \chi(\mathbb{T})$;

2) $\Delta_j$ is locally nilpotent;

3) there exists $t_j \in \mathbb{T}$ such that $\omega_j:=t^{\lambda_j}_j \in \mathbf{k}$ is not a root of unity and
\[ \sigma_{j}=t_j \cdot |_{S_{j-1}}.\]
As the $\mathbb{T}$-action on $S$ is determined by $\lambda_1, \ldots, \lambda_n$, we also denote the algebra $S$ by
\[
S=\mathbf{k}[X_1,\sigma_1,\Delta_1]\cdots [X_n;\sigma_n,\Delta_n]_{(\lambda_1, \ldots, \lambda_n)}.
\]
\hfill $\diamond$
}
\end{definition}

A CGL extension is thus a sequence
\[
\mathbf{k} = S_0 \subset S_1 \subset \cdots \subset S_n = S
\]
of associative algebras such that $S_j = S_{j-1}[X_j;\sigma_j,\Delta_j]_{\lambda_j}$ is a Cauchon extension of $S_{j-1}$ for $j \in [1,n]$.

\begin{definition}
\cite{QQ} \cite[$\S$2.1]{QO} \cite[$\S$1.3]{QQ2}
An element $r$ in a domain $R$ is called a prime element if $r$ is a normal element (i.e., $rR=Rr$), and the principal ideal $rR$ is a complete prime ideal (i.e., $R/(rR)$ is a domain).
\end{definition}

A main tool in the study of CGL extensions by K. Goodearl and M. Yakimov is the set of homogeneous prime elements of the nested
sequence of CGL extensions
\[
S_1 \subset \cdots \subset S_n = S.
\]
For a CGL extension $S=\mathbf{k}[X_1,\sigma_1,\Delta_1]\cdots [X_n;\sigma_n,\Delta_n]_{(\lambda_1, \ldots, \lambda_n)}$, and for $j \in [1, n]$, let $P_j \subset S_j$ be the set of all homogeneous prime elements of $S_j$, and let $Q_j \subset P_j$ be the set of elements in $P_j$ that are not in $S_{j-1}$. Clearly both $P_j$ and $Q_j$ are invariant under multiplication by scalars in
$\mathbf{k}^\times$. Define $P_0 = Q_0 = \mathbf{k}^\times$. The following theorem summarizes some results from \cite{QQ}, \cite[$\S$3 $\S$4]{QO}, and \cite[$\S$3]{QQ2}.

\begin{theorem}
\label{HP-CGL}
Let $S=\mathbf{k}[X_1,\sigma_1,\Delta_1]\cdots [X_n;\sigma_n,\Delta_n]_{(\lambda_1, \ldots, \lambda_n)}$ be a CGL extension. For $j \in [1,n]$, let $\omega_j$ be as in 3) of Definition \ref{T-CGL}. Then for each $j \in [1, n]$, $|Q_j/\mathbf{k}^\times|=1$, and
\[
P_j =Q_j \sqcup \bigsqcup_{i \in I_j} Q_i,
\]
where $I_j =  \{1 \leq i \leq j-1: \ \Delta_{i+1}|_{Q_i} = \cdots = \Delta_j|_{Q_i} = 0\}$. Moreover, there is a unique sequence $Y_S = \{Y_1,...,Y_n\}$ with $Y_j \in Q_j$ for $j \in [1, n]$, which is determined inductively as follows: $Y_0 = 1 \in Q_0$, and for $j \geq 1$,
\[
 Y_j=Y_{p(j)}Y_j-C_j=Y_{p(j)}X_j+(1-\omega_{j})^{-1}(\Delta_{j} \circ \sigma_{j}^{-1})(Y_{p(j)})=Y_{p(j)}(X_j-D_j),
\]
where $D_j$ is the distinguished element with respect to $S_{j-1}[X_j;\sigma_j,\Delta_j]_{\lambda_j}$. Here, $p(j)=0$ when $\Delta_j=0$, and when $\Delta_j \neq 0$, $p(j)$ is an (necessarily unique) integer in $[1, j-1]$ such that $y_{p(j)} \in P_{j-1}$ and $\Delta_j(y_{p(j)}) \neq 0$.
\end{theorem}

\begin{definition}
\label{qY_S}
{\rm
We refer to the sequence $Y_S = \{Y_1, \ldots, Y_n\}$ in Theorem \ref{HP-CGL} as the sequence of homogeneous prime elements of $S$.
\hfill $\diamond$
}
\end{definition}

\begin{definition}
\label{qlevel}
{\rm
Note that the map $p$ has the same properties as that in Theorem \ref{HPP-Poi-CGL}. We define the level sets of $S$ as in Definition \ref{level}.
\hfill $\diamond$
}
\end{definition}

Let $S=\mathbf{k}[X_1,\sigma_1,\Delta_1]\cdots [X_n;\sigma_n,\Delta_n]_{(\lambda_1, \ldots, \lambda_n)}$ be a CGL extension with the sequence of homogeneous prime elements $Y_S = \{Y_1, \ldots, Y_n\}$. Note that the map $p$, defined as in Theorem \ref{HP-CGL} for $S$, is injective when restricted to $\{ j \in [1, n]: \Delta_j \neq 0\}$. The following lemma follows directly from Theorem \ref{HP-CGL}.

\begin{lemma}
\label{qone-one}
1) For each level set $K$, let $k_{max}$ be the largest integer in $K$. Then $Y_{k_{max}}$ is a homogeneous prime element of $S$. Conversely, any homogeneous prime element of $S$ is a scalar multiple of $Y_{k_{max}}$ for some level set $K$. 

2) Define the $\mathbb{N}^n$-valued degrees for elements in $S$ as in $\S$\ref{subsec-not}. For $j \in [1,n]$, let $f_j=(f_{j,1},...,f_{j,n}) \in \mathbb{N}^n$ be the degree of $Y_{j}$ and let $K_j$ be the level set that contains $j$. For $j,l \in [1,n]$, $f_{j,l}=1$ if $l \in [1,j] \cap K_j$ and $f_{j,l}=0$ otherwise..
\end{lemma}

We state some result from \cite[Theorem 1.2]{QO} \cite[Proposition 3.11]{QQ2} in the following lemma:

\begin{lemma}
\label{log-can-qy}
For $i,j \in [1,n]$, $Y_iY_j=q_{i,j} Y_jY_i$, for some $q_{i,j} \in \mathbf{k}$.
\end{lemma}

By Theorem \ref{HP-CGL}, for $j \in [1,n]$, $Y_j = Y_{p(j)}X_j - C_j$, where $C_j \in S_{j-1}$. It follows that the algebra $S$ is contained in the quantum $\mathbf{k}$-torus generated by $Y_S$, i.e.,
\[ S \subset \Gamma_{\mathbf{q}}=\frac{\mathbf{k}<Y^{\pm 1}_1,...,Y^{\pm 1}_n>}{(Y_iY_j-q_{i,j}Y_jY_i)}. \]

\begin{definition}
\label{qtor}
{\rm
The algebra $\Gamma_{\mathbf{q}}$ is called the quantum torus of $S$.
\hfill $\diamond$
}
\end{definition}

Similar to the Poisson case, we have the following lemma-notation.

\begin{lemma-notation}
\label{qmon}
For $k \in [1,n]$ such that $p(k) \neq 0$ (or equivalently $\Delta_k \neq 0$), let $\Gamma_{\mathbf{q},k-1}$ be the quantum $\mathbf{k}$-torus generated by $Y_1,...,Y_{k-1}$ and let $\Gamma'_{\mathbf{q},k-1}$ be the quantum $\mathbf{k}$-torus generated by $Y_1,...,Y_{p(k)-1},Y_{p(k)+1},...,Y_{k-1}$. Then $C_k \in \Gamma'_{\mathbf{q},k-1}[Y_{p(k)}]$, the constant term of which, denoted by $B_k$, is a nonzero monomial in $\Gamma'_{\mathbf{q},k-1}$.
\end{lemma-notation}

\begin{proof}
For $j \in [1,n]$, let $\omega_j$ be as in 3) of Definition \ref{T-CGL}. By Theorem \ref{HP-CGL}, for $l \in [1,k-1]$,
\[ X_l=Y^{-1}_{p(l)}(Y_l+C_l), \ \text{where} \ C_l \in S_{l-1}. \]
Recall that the map $p$, when restricted to $\{ k \in [1, n]: \Delta_k \neq 0\}$, is injective. Since $\Delta_k \neq 0$, $p(k) \neq p(l)$ for any $l \in [1,k-1]$. As $S_{k-1}$ is generated by those $X_1,...,X_{k-1}$, $S_{k-1} \subset \Gamma'_{\mathbf{q},k-1}[Y_{p(k)}]$. Because $C_k \in S_{k-1}$, $C_k \in \Gamma'_{\mathbf{q},k-1}[Y_{p(k)}]$. 

Since $\Delta_k \neq 0$, by Theorem \ref{HP-CGL}, $C_k \neq 0$. Assume $B_k=0$. Then $C_k=Y_{p(k)}s_1s^{-1}_2$, where $s_1,s_2 \in S_{k-1}$ and $s_2 \notin Y_{p(k)}S_{k-1}$. One sees that $C_ks_2 \in Y_{p(k)}S_{k-1}$. As $Y_{p(k)}$ is prime in $S_{k-1}$ and $s_2 \notin Y_{p(k)}S_{k-1}$, one has $C_k \in Y_{p(k)}S_{k-1}$, which contradicts to $Y_k=Y_{p(k)}X_k+C_k$ being prime in $S_k$. Thus, $B_k \neq 0$.

By Theorem \ref{HP-CGL}, the distinguished element with respect to $S_{k-1}[X_k;\sigma_k,\Delta_k]_{\lambda_k}$ is
\[ D_k=Y^{-1}_{p(k)}C_k. \]
Since $B_k$ is the constant term of $C_k \in \Gamma'_{\mathbf{q},k-1}[Y_{p(k)}]$, $Y^{-1}_{p(k)}B_k$ is a monomial term of $D_k \in \Gamma'_{\mathbf{q},k-1}[Y^{\pm 1}_{p(k)}]$ and
\[ D_k-Y^{-1}_{p(k)}B_k \in \Gamma'_{\mathbf{q},k-1}[Y_{p(k)}]. \]
For $l \in [1,k-1]$, by (\ref{qdis1}),
\[ D_kY_l-\sigma_k(Y_l)D_k=\Delta_k(Y_l) \in S_{k-1} \subset \Gamma'_{\mathbf{q},k-1}[Y_{p(k)}]. \]
It is easy to see that $\Gamma'_{\mathbf{q},k-1}[Y_{p(k)}]$ is a subalgebra of $\Gamma_{\mathbf{q},k-1}$, so for $l \in [1,k-1]$,
\[ (D_k-Y^{-1}_{p(k)}B_k)Y_l-\sigma_k(Y_l)(D_k-Y^{-1}_{p(k)}B_k) \in \Gamma'_{\mathbf{q},k-1}[Y_{p(k)}]. \]
Hence, for $l \in [1,k-1]$,
\[ (Y^{-1}_{p(k)}B_k)Y_l-\sigma_k(Y_l)(Y^{-1}_{p(k)}B_k) \in \Gamma'_{\mathbf{q},k-1}[Y_{p(k)}]. \]
Therefore,
\begin{equation}
\label{equ-qmon-1}
(Y^{-1}_{p(k)}B_k)Y_l=\sigma_k(Y_l)(Y^{-1}_{p(k)}B_k), \ \ l \in [1,k-1] \setminus \{p(k)\}.
\end{equation}
By Theorem \ref{HP-CGL}, one has $\Delta_k(Y_{p(k)})=(\omega_k-1)\sigma_{k}(Y_{p(k)})D_k$. By (\ref{qdis1}),
\[ D_kY_{p(k)}=\omega_k\sigma_{k}(Y_{p(k)})D_k. \]
As $Y^{-1}_{p(k)}B_k$ is a monomial term of $D_k \in \Gamma'_{\mathbf{q},k-1}[Y^{\pm 1}_{p(k)}]$,
\begin{equation}
\label{equ-qmon-2}
(Y^{-1}_{p(k)}B_k)Y_{p(k)}=\omega_k\sigma_{k}(Y_{p(k)})(Y^{-1}_{p(k)}B_k).
\end{equation}
Any monomial term of $B_k \in \Gamma'_{\mathbf{q},k-1}$ satisfies (\ref{equ-qmon-1}) and (\ref{equ-qmon-2}) by replacing $B_k$. By \cite[Lemma 8.14]{QQ2}, $B_k$ has to be a monomial in $\Gamma'_{\mathbf{q},k-1}$.
\end{proof}

\subsection{Symmetric CGL extensions}\label{subsec-sym-CGL}

\begin{definition}
\cite[Definition 3]{QQ} \cite[Definition 6.2]{QO} \cite[Definition 3.12]{QQ2}
\label{qsym}
{\rm
Let $S=\mathbf{k}[X_1;\sigma_1,\Delta_1] \cdots [X_n;\sigma_n,\Delta_n]_{(\lambda_1,...,\lambda_n)}$ be a CGL extension. Then $S$ is said to be symmetric if

1) for $1 \leq i <j \leq n$, $\Delta_j(X_i) \in S_{[i+1,j-1]}$;

2) there exist $t'_1,...,t'_n \in \mathbb{T}$ such that

\indent \indent a) $\sigma_j(X_i)=(t'_i)^{-\lambda_j}X_i$ for $1 \leq i < j \leq n$, 

\indent \indent b) $(t'_j)^{\lambda_j} \in \mathbf{k}$ is not a root of unity for $j \in [1,n]$.\\
Here, $S_{[i+1,j-1]}$ is the subalgebra of $S$ generated by $X_{i+1},...,X_{j-1}$.
\hfill $\diamond$
}
\end{definition}

Let $S'=\mathbf{k}[X_0;\sigma_0,\Delta_0][X_1;\sigma_1,\Delta_1]\cdots [X_n;\sigma_n,\Delta_n]_{(\lambda_0,\lambda_1,...,\lambda_n)}$ be a symmetric CGL extension. Then the subalgebra
\[ S:=\mathbf{k}[X_1;\sigma_1,\Delta_1]\cdots [X_n;\sigma_n,\Delta_n]_{(\lambda_1,...,\lambda_n)} \]
is also a symmetric CGL extension. Let $p$ be the map for $S$ as in Theorem \ref{HP-CGL}. Let $Y_S=(Y_1,...,Y_n)$ be the sequence of homogeneous prime elements of $S$ and let $\Gamma_{\mathbf{q}}$ be the quantum torus of $S$.

Note that $S'$ can be viewed as a Cauchon extension $S'=S[X_0;\sigma,\Delta]_{\lambda_0}$ of $S$, where
\[ \sigma(X_j)X_0=X_j\sigma^{-1}_j(X_0) \ \ \text{and} \ \ \Delta(X_j)=-(\Delta_j \circ \sigma^{-1}_j)(X_0), \ \ j \in [1,n].\]
Let $t_0,t_1,...,t_n$ be in $\mathbb{T}$ such that they satisfy 3) of Definition \ref{T-CGL} for $S'$. Let $t'_0,t'_1,...,t'_n$ be in $\mathbb{T}$ such that they satisfy 2) of Definition \ref{qsym} for $S'$. Then for $j \in [1,n]$, $S_{j}[X_0;\sigma,\Delta]_{\lambda_0}$ (here $\sigma$ and $\Delta$ are their restrictions on $S_j$) is a Cauchon extension of $S_j$, for which $t'_0$ satisfies 2) of Definition \ref{T-O}. Furthermore,
\[ \mathbf{k}[X_1;\sigma_1,\Delta_1]\cdots [X_j;\sigma_j,\Delta_j][X_0;\sigma,\Delta]_{(\lambda_1,...,\lambda_j,\lambda_0)} \]
is a CGL extension, for which $t_1,...,t_j,t'_0$ satisfy 3) of Definition \ref{T-CGL}.

\begin{lemma}
\label{qk-p(k)}
For $k \in [1,n]$ such that $p(k) \neq 0$, $\Delta(Y_k)=0$ if and only if $\Delta(Y_{p(k)})=0$.
\end{lemma}

\begin{proof}
1) By Theorem \ref{HP-CGL}, $Y_k=Y_{p(k)}X_k-C_k$, where $C_k \in S_{k-1}$. Then
\[ \Delta(Y_k)=\Delta(Y_{p(k)}X_k-C_k)=\Delta(Y_{p(k)})X_k+\sigma(Y_{p(k)})\Delta(X_k)-\Delta(C_k). \]
Since $S'$ is symmetric, $\sigma(Y_{p(k)})\Delta(X_k)-\Delta(C_k) \in S_{k-1}$. Then $\Delta(Y_k)=0$ would imply $\Delta(Y_{p(k)})=0$.

2) Conversely, assume $\Delta(Y_{p(k)})=0$ but $\Delta(Y_k) \neq 0$. Recall that $P_{k}$ and $P_{k-1}$ are the sets of all homogeneous prime elements of $S_k$ and $S_{k-1}$ respectively. Applying Theorem \ref{HP-CGL} to
\[ \mathbf{k}[X_1;\sigma_1,\Delta_1] \cdots [X_k;\sigma_k,\Delta_k][X_0;\sigma,\Delta]_{(\lambda_1,...,\lambda_k,\lambda_0)}, \]
one sees that $\Delta$ vanishes on $P_k \setminus (\mathbf{k}^{\times}Y_k)$. Still by Theorem \ref{HP-CGL},
\[ P_{k-1}=(P_k \setminus (\mathbf{k}^{\times}Y_k)) \cup (\mathbf{k}^{\times}Y_{p(k)}). \]
Since $\Delta(Y_{p(k)})=0$ by assumption, $\Delta(P_{k-1})=0$. Applying Theorem \ref{HP-CGL} to
\[ \mathbf{k}[X_1;\sigma_1,\Delta_1] \cdots [X_{k-1};\sigma_{k-1},\Delta_{k-1}][X_0;\sigma,\Delta]_{(\lambda_1,...,\lambda_{k-1},\lambda_0)},\]
one has $\Delta=0$ on $S_{k-1}$ and $X_0,Y_{p(k)}$ are both homogeneous prime elements of $S_{k-1}[X_0]$. Applying Theorem \ref{HP-CGL} to
\[ \mathbf{k}[X_0;\sigma_0,\Delta_0][X_1;\sigma_1,\Delta_1] \cdots [X_{k};\sigma_{k},\Delta_{k}]_{(\lambda_0,\lambda_1,...,\lambda_{k})},\]
one has $\Delta_k(X_0)=0$ since $\Delta_k(Y_{p(k)}) \neq 0$. It follows from $\Delta(X_k)=-(\Delta_k \circ \sigma^{-1}_k)(X_0)$ that $\Delta(X_k)=0$. Then $\Delta=0$ on $S_k$, contradicting to $\Delta(Y_k) \neq 0$. Therefore, $\Delta(Y_{p(k)})=0$ would imply $\Delta(Y_{k})=0$.
\end{proof}

In the rest of this section, assume $\Delta \neq 0$. By Theorem \ref{HP-CGL}, there exists a unique $k \in [1,n]$ such that $\Delta(Y_k) \neq 0$ and $Y_k$ is a homogeneous prime element of $S$.
Let $\omega$ be as in 2) of Definition \ref{T-O} for $S[X_0;\sigma,\Delta]_{\lambda_0}$. For $j \in [1,n]$, let $\omega_j$ be as in 3) of Definition \ref{T-CGL} for $S$. Let $m$ be the largest nonnegative integer that $p^{m}(k) \in [1,n]$ and define
\[ D^{(m+1)}=0 \ \ \text{and} \ \ D^{(i)}=-(1-\omega)^{-1}Y^{-1}_{p^i(k)}(\Delta \circ \sigma^{-1})(Y_{p^i(k)}) \in \Gamma_{\bf{q}}. \]
By Theorem \ref{HP-CGL}, $D^{(i)}$ is the distinguished element with respect to $S_{p^{i-1}(k)-1}[X_0;\sigma,\Delta]_{\lambda_0}$ ($p^{-1}(k):=n+1$), for $i \in [1,m]$.

\begin{lemma}
\label{qd-M}
1) Let $i \in [0,m]$ and one has
\[ D^{(i)}-D^{(i+1)}=Y^{-1}_{p^i(k)}Y^{-1}_{p^{i+1}(k)}\mathcal{M}^{(i)}, \]
where $\mathcal{M}^{(i)}$ is either a nonzero constant or a product of homogeneous prime elements of $S_{p^{i}(k)-1}$.\\
2) Let $i \in [0,m-1]$ and let $B_{p^{i}(k)}$ be as in Lemma-Notation \ref{qmon}. One has
\[ D^{(i)}-D^{(i+1)}=\omega_{p^{i}(k)}Y^{-1}_{p^i(k)}B_{p^{i}(k)}(D^{(i+1)}-D^{(i+2)}). \]
\end{lemma}

\begin{proof}
1) Define $g_i=D^{(i)}-D^{(i+1)}$. Because $D^{(i)}$ is the distinguished element with respect to $S_{p^{i-1}(k)-1}[X_0;\sigma,\Delta]_{\lambda_0}$ and $S_{p^{i}(k)-1} \subset S_{p^{i-1}(k)-1}$,
\[ D^{(i)}s=\sigma(s)D^{(i)}+\Delta(s), \ \ s \in S_{p^{i}(k)-1}. \]
Because $D^{(i+1)}$ is the distinguished element with respect to $S_{p^{i}(k)-1}[X_0;\sigma,\Delta]_{\lambda_0}$,
\[ D^{(i+1)}s=\sigma(s)D^{(i+1)}+\Delta(s), \ \ s \in S_{p^{i}(k)-1}. \]
Thus, $g_is=\sigma(s)g_i$ for $s \in S_{p^{i}(k)-1}$, which means $g_i S_{p^{i}(k)-1} = S_{p^{i}(k)-1} g_i$. By Theorem \ref{HP-CGL}, $Y_{p^{i}(k)}=Y_{p^{i+1}(k)}(X_{p^{i}(k)}-D_{p^{i}(k)})$ and $Y_{p^{i+1}(k)}$ is normal in $S_{p^{i}(k)-1}$. Hence,
\[ Y_{p^{i}(k)} S_{p^{i}(k)-1} =  S_{p^{i}(k)-1} Y_{p^{i}(k)} \ \ \text{and} \ \ Y_{p^{i+1}(k)} S_{p^{i}(k)-1} = S_{p^{i}(k)-1} Y_{p^{i+1}(k)}. \]
Define $\mathcal{M}^{(i)}=Y_{p^{i+1}(k)}Y_{p^{i}(k)}g_i$. Then
\[ \mathcal{M}^{(i)} S_{p^{i}(k)-1} =  S_{p^{i}(k)-1} \mathcal{M}^{(i)}.\]
Through direct computation, one sees that $\mathcal{M}^{(i)} \in S_{p^{i}(k)-1}$ since $S'$ is symmetric. It follows that $\mathcal{M}^{(i)}$ is a homogeneous normal element of $S_{p^{i}(k)-1}$. By \cite[Proposition 3.2, Theorem 3.7]{LLR} and \cite[Proposition 2.2]{QO}, $\mathcal{M}^{(i)}$ is either a constant or a product of homogeneous prime elements of $S_{p^{i}(k)-1}$. Suppose $\mathcal{M}^{(i)}=0$. Then $D^{(i)}-D^{(i+1)}=0$. Through direct computation, one sees that $Y_{p^i(k)}\Delta(Y_{p^{i+1}(k)})$ is a scalar multiple of $Y_{p^{i+1}(k)}\Delta(Y_{p^i(k)})$. Since $Y_{p^i(k)}$ is prime in $S_{p^{i}(k)}$ and clearly $Y_{p^{i+1}(k)} \notin Y_{p^i(k)}S_{p^{i}(k)}$, $\Delta(Y_{p^i(k)}) \in Y_{p^i(k)}S_{p^{i}(k)}$. Applying Theorem \ref{HP-CGL} to
\[ \mathbf{k}[X_1;\sigma_1,\Delta_1] \cdots [X_{p^i(k)};\sigma_{p^i(k)},\Delta_{p^i(k)}][X_0;\sigma,\Delta]_{(\lambda_1,...,\lambda_{p^i(k)},\lambda_0)}, \]
as $\Delta(Y_{p^i(k)}) \neq 0$ by Lemma \ref{qk-p(k)}, one sees that $Y_{p^i(k)}X_0+(1-\omega)^{-1}Y^{-1}_{p^i(k)}(\Delta \circ \sigma^{-1})(Y_{p^i(k)})$ is prime in $S_{p^i(k)}[X_0]$, contradicting to $\Delta(Y_{p^i(k)}) \in Y_{p^i(k)}S_{p^{i}(k)}$. Thus, $\mathcal{M}^{(i)}$ is nonzero. As $g_i=D^{(i)}-D^{(i+1)}$ and $\mathcal{M}^{(i)}=Y_{p^{i+1}(k)}Y_{p^{i}(k)}g_i$,
\[ D^{(i)}-D^{(i+1)}=Y^{-1}_{p^i(k)}Y^{-1}_{p^{i+1}(k)}\mathcal{M}^{(i)}. \]

2) Let $i \in [0,m-1]$. one has
\begin{align*}
&\mathcal{M}^{(i)}=Y_{p^{i+1}(k)}Y_{p^{i}(k)}(D^{(i)}-D^{(i+1)})\\
&=Y_{p^{i+1}(k)}Y_{p^{i}(k)}D^{(i)}-Y^2_{p^{i+1}(k)}X_{p^{i}(k)}D^{(i+1)}+Y_{p^{i+1}(k)}C_{p^{i}(k)}D^{(i+1)}\\
&=Y_{p^{i+1}(k)}Y_{p^{i}(k)}D^{(i)}-Y^2_{p^{i+1}(k)}\sigma_{p^{i}(k)}(D^{(i+1)})X_{p^{i}(k)}+\omega_{p^{i}(k)}Y_{p^{i+1}(k)}C_{p^{i}(k)}D^{(i+1)}.
\end{align*}
It follows from
\[ Y_{p^{i+1}(k)}Y_{p^{i}(k)}D^{(i)}-Y^2_{p^{i+1}(k)}\sigma_{p^{i}(k)}(D^{(i+1)})X_{p^{i}(k)} \in Y_{p^{i+1}(k)}S_{p^{i}(k)-1} \]
that
\begin{equation}
\label{qdegree}
\mathcal{M}^{(i)}-\omega_{p^{i}(k)}Y_{p^{i+1}(k)}C_{p^{i}(k)}D^{(i+1)} \in Y_{p^{i+1}(k)}S_{p^{i}(k)-1}.
\end{equation}
Recall from Lemma-Notation \ref{qmon} that $\Gamma_{\mathbf{q},p^{i}(k)-1}$ and $\Gamma'_{\mathbf{q},p^{i}(k)-1}$ are the quantum $\mathbf{k}$-tori generated by $Y_1,...,Y_{p^{i}(k)-1}$ and $Y_1,...,Y_{p^{i+1}(k)-1},Y_{p^{i+1}(k)+1},...,Y_{p^{i}(k)-1}$ respectively.
It is easy to see that
\[ \Gamma_{\mathbf{q},p^{i}(k)-1}=\Gamma'_{\mathbf{q},p^{i}(k)-1}[Y^{\pm 1}_{p^{i+1}(k)}]. \]
Now consider every element in $\Gamma_{\mathbf{q},p^{i}(k)-1}$, including those involved in (\ref{qdegree}), as an element in $\Gamma'_{\mathbf{q},p^{i}(k)-1}[Y^{\pm 1}_{p^{i+1}(k)}]$. The lowest degree term of $C_{p^{i}(k)}$ is $B_{p^{i}(k)}$ with degree 0. By part 1), the lowest degree term of $D^{(i+1)}$ is $D^{(i+1)}-D^{(i+2)}$ with degree -1. Thus, the lowest degree term of $\omega_{p^{i}(k)}Y_{p^{i+1}(k)}C_{p^{i}(k)}D^{(i+1)}$ is $\omega_{p^{i}(k)}Y_{p^{i+1}(k)}B_{p^{i}(k)}(D^{(i+1)}-D^{(i+2)})$ with degree 0. By (\ref{qdegree}),
the lowest degree term of $\mathcal{M}^{(i)}-\omega_{p^{i}(k)}Y_{p^{i+1}(k)}C_{p^{i}(k)}D^{(i+1)}$ has degree greater than 0. By 1) of Lemma \ref{qone-one}, $\mathcal{M}^{(i)}$ is a monomial in $\Gamma_{\mathbf{q},p^{i}(k)-1}$ because $\mathcal{M}^{(i)}$ is either a nonzero constant or a product of homogeneous prime elements of $S_{p^{i}(k)-1}$. Hence,
\[ \mathcal{M}^{(i)}=\omega_{p^{i}(k)}Y_{p^{i+1}(k)}B_{p^{i}(k)}(D^{(i+1)}-D^{(i+2)}). \]
Therefore, by part 1),
\[ D^{(i)}-D^{(i+1)}=\omega_{p^{i}(k)}Y^{-1}_{p^i(k)}B_{p^{i}(k)}(D^{(i+1)}-D^{(i+2)}). \]
\end{proof}

\section{Quantization of Poisson-CGL Extensions}\label{sec-qua}

Fix a field $\mathbf{k}$ of characteristic $0$. Let $L=\mathbf{k}[q^{\pm 1}]$ be the $\mathbf{k}$-algebra of Laurent polynomials in the single variable $q$ and let $\mathbb{K}=\mathbf{k}(q)$ be its fraction field. Let $\mathbb{T}$ be a split $\mathbf{k}$-torus. We follow the notation in $\S$\ref{subsec-not}.

\subsection{Quantum-CGL extensions}\label{subsec-qua-CGL}

In \cite[$\S$4]{LL}, the semiclassical process of a special type of Ore extensions over $L$ is studied. This type of Ore extensions can be defined as follows:

\begin{definition}
{\rm
Let $\mathcal{A}$ be a quantum $L$-algebra (see Definition \ref{qua-alg}). Let $\mathcal{B}=\mathcal{A}[X;\sigma,\Delta]$ be an Ore extension of $\mathcal{A}$ such that $\sigma$ and $\Delta$ are $L$-linear maps. Then $\mathcal{B}$ is called a quantum-Ore extension of $\mathcal{A}$ if $\mathcal{B}$ is also a quantum $L$-algebra.
\hfill $\diamond$
}
\end{definition}

Note that $\mathcal{B}$ is always free as an $L$-module and $B:=\mathcal{B}/(q-1)\mathcal{B}$ is commutative if and only if $\sigma(a)-a$ and $\Delta(a)$ lie in $(q-1)\mathcal{A}$ for any $a \in \mathcal{A}$. Through the semiclassical process, we obtain Poisson-Ore extensions from quantum-Ore extensions.

\begin{proposition}
\cite[Proposition 4.1]{LL}
\label{Q-O-P-O}
Let $\mathcal{B}=\mathcal{A}[X;\sigma,\Delta]$ be a quantum-Ore extension of a quantum $L$-algebra $\mathcal{A}$. Then $B:=\mathcal{B}/(q-1)\mathcal{B}$ is a Poisson-Ore extension of the form $B=A[x;\theta,\delta]$, where $A=\mathcal{A}/(q-1)\mathcal{A}$ and $x=X+(q-1)\mathcal{B} \in B$. More precisely, for any $a \in \mathcal{A}$,
\begin{align*}
&\theta(a+(q-1)\mathcal{A})=\frac{\sigma(a)-a}{q-1}+(q-1)\mathcal{A}, \\
&\delta(a+(q-1)\mathcal{A})=\frac{\Delta(a)}{q-1}+(q-1)\mathcal{A}.
\end{align*}
\end{proposition}

An iterated quantum-Ore extension thus gives rise to an iterated Poisson-Ore extension as its semiclassical limit by repeatedly applying Proposition \ref{Q-O-P-O}.

\begin{definition}
\label{ite-qua}
{\rm
We say an iterated quantum-Ore extension
\[ \mathcal{A}=L[X_1;\sigma_1,\Delta_1]\cdots [X_n;\sigma_n,\Delta_n] \]
is a quantization of an iterated Poisson-Ore extension
\[ A'=\mathbf{k}[x'_1;\theta'_1,\delta'_1]\cdots [x'_n;\theta'_n,\delta'_n] \]
if there is a Poisson algebra isomorphism from $A'$ to $A=\mathcal{A}/(q-1)\mathcal{A}$ by sending $x'_i$ to $x_i=X_i+(q-1)\mathcal{A}$, for $i \in [1,n]$.
\hfill $\diamond$
}
\end{definition}

For $j \in [1,n]$, denote by $\mathcal{A}_{j}$ the subalgebra of $\mathcal{A}$ generated by $X_1,...,X_{j}$ and set $\mathcal{A}_0 = L$.

\begin{definition}
\label{Q-CGL}
{\rm
Let $\mathcal{A}=L[X_1;\sigma_1,\Delta_1]\cdots [X_n;\sigma_n,\Delta_n]$ be an iterated quantum-Ore extension with a rational $\mathbb{T}$-action by $\mathbf{k}$-algebra automorphisms. Then $\mathcal{A}$ is called a quantum-CGL extension if for $j \in [1,n]$,

1) $X_j$ is homogeneous with weight $\lambda_j \in \chi(\mathbb{T})$;

2) $\Delta_j$ is locally nilpotent;

3) there exists $h_j \in \chi_{\ast}(\mathbb{T})$ satisfying $\eta_j:=\lambda_{j}(h_{j}) \neq 0$ and
\[ \sigma_{j}(X_i)=q^{\lambda_i(h_j)}X_i, \ \ i \in [1,j-1].\]
We denote this quantum-CGL extension by
\[ \mathcal{A}=L[X_1;\sigma_1,\Delta_1]\cdots [X_n;\sigma_n,\Delta_n]_{(\lambda_1, \ldots, \lambda_n)}. \]
\hfill $\diamond$
}
\end{definition}

Through direct checking, one sees that the semiclassical limit of a quantum-CGL extension is a Poisson-CGL extension (see Definition \ref{T-Poi-CGL}). Condition 3) of Definition \ref{Q-CGL} results in an "integral" condition on the semiclassical limit of a quantum-CGL extension.

\begin{definition}
\label{integral}
{\rm
Let $A=\mathbf{k}[x_1;\theta_1,\delta_1] \cdots [x_n;\theta_n,\delta_n]_{(\lambda_1,...,\lambda_n)}$ be a Poisson-CGL extension. Then $A$ is said to be integral if there exist $h_1,...,h_n \in \chi_{\ast}(\mathbb{T})$ such that for $j \in [1,n]$, $\eta_j:=\lambda_{j}(h_{j}) \neq 0$ and $\theta_{j}=\partial h_{j}|_{A_{j-1}}$.
\hfill $\diamond$
}
\end{definition}

We also have the notion of "symmetric" for integral Poisson-CGL extensions and quantum-CGL extensions.

\begin{definition}
\label{P-sym}
{\rm
Let $A=\mathbf{k}[x_1;\theta_1,\delta_1] \cdots [x_n;\theta_n,\delta_n]_{(\lambda_1,...,\lambda_n)}$ be an integral Poisson-CGL extension. Then $A$ is said to be symmetric if

1) for $1 \leq i <j \leq n$, $\delta_j(x_i) \in A_{[i+1,j-1]}$;

2) there exist $h'_1,...,h'_n \in \chi_{\ast}(\mathbb{T})$ such that

\indent \indent a) $\theta_j(x_i)=-\lambda_j(h'_i)x_i$ for $1 \leq i < j \leq n$, 

\indent \indent b) $\lambda_{j}(h'_{j}) \neq 0$ for $j \in [1,n]$.\\
Here, $A_{[i+1,j-1]}$ is the subalgebra of $A$ generated by $x_{i+1},...,x_{j-1}$.
\hfill $\diamond$
}
\end{definition}

\begin{definition}
\label{Q-sym}
{\rm
Let $\mathcal{A}=L[X_1;\sigma_1,\Delta_1]\cdots [X_n;\sigma_n,\Delta_n]_{(\lambda_1, \ldots, \lambda_n)}$ be a quantum-CGL extension. Then $\mathcal{A}$ is said to be symmetric if 

1) for $1 \leq i <j \leq n$, $\Delta_j(X_i) \in \mathcal{A}_{[i+1,j-1]}$;

2) there exist $h'_1,...,h'_n \in \chi_{\ast}(\mathbb{T})$ such that

\indent \indent a) $\sigma_j(X_i)=q^{-\lambda_j(h'_i)}X_i$ for $1 \leq i < j \leq n$,

\indent \indent b) $\lambda_{j}(h'_{j}) \neq 0$ for $j \in [1,n]$.\\
Here, $\mathcal{A}_{[i+1,j-1]}$ is the subalgebra of $\mathcal{A}$ generated by $X_{i+1},...,X_{j-1}$.
\hfill $\diamond$
}
\end{definition}

\begin{remark}
{\rm
\label{extension}
For a quantum-CGL extension $\mathcal{A}$, denote by $\mathcal{A}^{ex}$ the extended algebra $\mathcal{A} \otimes_{L} \mathbb{K}$, extending the base ring $L$ to the field $\mathbb{K}$. Define $\mathbb{T}_q=(\mathbb{K}^{\times})^r$ and identify its character group $\chi(\mathbb{T}_q)$ with $\chi(\mathbb{T})$. By assigning $\lambda_1,...,\lambda_n$ to $X_1,...,X_n$ respectively as their weights, we define a rational $\mathbb{T}_q$-action on $\mathcal{A}^{ex}$ by $\mathbb{K}$-algebra automorphisms. For any $h=(h_{(1)},...,h_{(r)}) \in \chi_{\ast}(\mathbb{T})$, denote by $q^{h}$ the element $(q^{h_{(1)}},...,q^{h_{(r)}}) \in \mathbb{T}_q$.  It is straightforward to check that $\mathcal{A}^{ex}$ is a CGL extension under the $\mathbb{T}_q$-action by letting $t_j=q^{h_j}$ for $j \in [1,n]$. Furthermore, if $\mathcal{A}$ is symmetric, then $\mathcal{A}^{ex}$ is symmetric by letting $t'_j=q^{h'_j}$ for $j \in [1,n]$.
\hfill $\diamond$
}
\end{remark}

It is easy to see that the semiclassical limit of a symmetric quantum-CGL extension is a symmetric integral Poisson-CGL extension.

\subsection{Quantization of symmetric integral Poisson-CGL extensions}\label{subsec-qua}

Given a symmetric integral Poisson-CGL extension, we construct a symmetric quantum-CGL extension, which is preferred in the following sense:

\begin{definition}
\label{per-qua}
{\rm
Let $A=\mathbf{k}[x_1;\theta_1,\delta_1] \cdots [x_n;\theta_n,\delta_n]_{(\lambda_1,...,\lambda_n)}$ be a symmetric integral Poisson-CGL extension. Then a symmetric quantum-CGL extension
\[ \mathcal{A}=L[X_1;\sigma_1,\Delta_1]\cdots [X_n;\sigma_n,\Delta_n]_{(\lambda_1,...,\lambda_n)} \]
(note that we require $x_j$ and $X_j$ have the same weight $\lambda_j$ for $j \in [1,n]$) is called a preferred quantization of $A$ if

1) $\mathcal{A}$ is a quantization of $A$ in the sense of Definition \ref{ite-qua};

2) $A$ and $\mathcal{A}^{ex}$ share the same level sets (see Definition \ref{level} and Definition \ref{qlevel});

3) $Y_{\mathcal{A}^{ex}}=(Y_1,...,Y_n)$, the sequence of homogeneous prime elements of $\mathcal{A}^{ex}$ (see Definition \ref{qY_S}), lie in $\mathcal{A}$.
\hfill $\diamond$
}
\end{definition}

Recall that an element $r$ in a domain $R$ is called a prime element if $r$ is a normal element (i.e., $rR=Rr$), and the principal ideal $rR$ is a complete prime ideal (i.e., $R/(rR)$ is a domain).

\begin{lemma}
\label{prime}
Let $\mathcal{A}=L[X_1;\sigma_1,\Delta_1]\cdots [X_n;\sigma_n,\Delta_n]_{(\lambda_1,...,\lambda_n)}$ be a quantum-CGL extension. Let $Y_{\mathcal{A}^{ex}}=(Y_1,...,Y_n)$ be the sequence of homogeneous prime elements of $\mathcal{A}^{ex}$. Assume $Y_{\mathcal{A}^{ex}}$ lie in $\mathcal{A}$. Then for $j \in [1,n]$, $Y_j$ is a prime element of $\mathcal{A}_{p^{-1}(j)-1}$ (note that $p^{-1}(j)$ may not exist, in which case, $p^{-1}(j):=n+1$).
\end{lemma}

\begin{proof}
Let $j \in [1,n]$. We claim that
\begin{align*}
&\mathcal{A}_{p^{-1}(j)-1} \cap (\mathcal{A}_{p^{-1}(j)-1}^{ex} \ast Y_j)=\mathcal{A}_{p^{-1}(j)-1} \ast Y_j, \ \text{and} \\
&\mathcal{A}_{p^{-1}(j)-1} \cap (Y_j \ast \mathcal{A}_{p^{-1}(j)-1}^{ex})=Y_j \ast \mathcal{A}_{p^{-1}(j)-1}.
\end{align*}
Define the $\mathbb{N}^n$-valued degrees for elements in $\mathcal{A}^{ex}$ as in $\S$\ref{subsec-not}. From the definition of $Y_j$, the highest degree term of $Y_j$ is
\[ X_{p^{m_j}(j)} \ast X_{p^{m_j-1}(j)} \ast \cdots \ast X_{j}, \]
where $m_j$ is the largest nonnegative integer such that $p^{m_j}(j) \in [1,n]$. Denote by $f_j \in \mathbb{N}^n$ the degree of $X_{p^{m_j}(j)} \ast X_{p^{m_j-1}(j)} \ast \cdots \ast X_{j}$. Assume $a \in \mathcal{A}_{p^{-1}(j)-1}^{ex} \setminus \mathcal{A}_{p^{-1}(j)-1}$. Then $a=a_1+a_2$, where $a_1 \in \mathcal{A}_{p^{-1}(j)-1}$ and each monomial term of $a_2$ has coefficient in $\mathbb{K} \setminus L$. In order to prove the claim, it suffices to show that both $Y_j \ast a_2$ and $a_2 \ast Y_j$ lie in $\mathcal{A}_{p^{-1}(j)-1}^{ex} \setminus \mathcal{A}_{p^{-1}(j)-1}$. Let $a_h=\varepsilon X^{f_h}$ be the highest degree term of $a_2$, where $\varepsilon \in \mathbb{K} \setminus L$. By 3) of Definition \ref{Q-CGL}, the highest degree term of $Y_j \ast a_2$ and $a_2 \ast Y_j$ are $q^{l_1}\varepsilon X^{f_j+f_h}$ and $q^{l_2}\varepsilon X^{f_j+f_h}$ respectively, for some $l_1,l_2 \in \mathbb{Z}$. As $\varepsilon \in \mathbb{K} \setminus L$, both $q^{l_1}\varepsilon$ and $q^{l_2}\varepsilon$ lie in $\mathbb{K} \setminus L$. Therefore, both $Y_j \ast a_2$ and $a_2 \ast Y_j$ lie in $\mathcal{A}_{p^{-1}(j)-1}^{ex} \setminus \mathcal{A}_{p^{-1}(j)-1}$. The claim has been proven.

Now we prove that $Y_j$ is a prime element of $\mathcal{A}_{p^{-1}(j)-1}$. By Theorem \ref{HP-CGL}, $Y_j$ is a prime element of $\mathcal{A}^{ex}_{p^{-1}(j)-1}$. By definition, 

1) $Y_j \ast \mathcal{A}_{p^{-1}(j)-1}^{ex}=\mathcal{A}_{p^{-1}(j)-1}^{ex} \ast Y_j$. 

2) $Y_j \ast \mathcal{A}_{p^{-1}(j)-1}^{ex}$ is a complete prime ideal in $\mathcal{A}_{p^{-1}(j)-1}^{ex}$.\\
By the claim and the fact that $Y_j \ast \mathcal{A}_{p^{-1}(j)-1}^{ex}=\mathcal{A}_{p^{-1}(j)-1}^{ex} \ast Y_j$,
\[ Y_j \ast \mathcal{A}_{p^{-1}(j)-1}=\mathcal{A}_{p^{-1}(j)-1} \ast Y_j. \]
Suppose two nonzero elements $g_1,g_2 \in \mathcal{A}_{p^{-1}(j)-1}$ satisfy $g_1 \ast g_2 \in Y_j \ast \mathcal{A}_{p^{-1}(j)-1}$. Since $Y_j \ast \mathcal{A}_{p^{-1}(j)-1} \subset Y_j \ast \mathcal{A}_{p^{-1}(j)-1}^{ex}$, $g_1 \ast g_2 \in Y_j \ast \mathcal{A}_{p^{-1}(j)-1}^{ex}$. As $Y_j \ast \mathcal{A}_{p^{-1}(j)-1}^{ex}$ is a complete prime ideal in $\mathcal{A}_{p^{-1}(j)-1}^{ex}$, either $g_1 \in Y_j \ast \mathcal{A}_{p^{-1}(j)-1}^{ex}$ or $g_2 \in Y_j \ast \mathcal{A}_{p^{-1}(j)-1}^{ex}$. Without loss of generality, assume $g_1 \in Y_j \ast \mathcal{A}_{p^{-1}(j)-1}^{ex}$. As $g_1 \in \mathcal{A}_{p^{-1}(j)-1}$, by the claim,
\[ g_1 \in \mathcal{A}_{p^{-1}(j)-1} \cap (Y_j \ast \mathcal{A}_{p^{-1}(j)-1}^{ex})=Y_j \ast \mathcal{A}_{p^{-1}(j)-1}. \]
Thus, $Y_j \ast \mathcal{A}_{p^{-1}(j)-1}$ is a complete prime ideal of $\mathcal{A}_{p^{-1}(j)-1}$.
\end{proof}

We now proceed step by step to construct a preferred quantization of a symmetric integral Poisson-CGL extension. Given a symmetric integral Poisson-CGL extension
\[ B=\mathbf{k}[x_0;\theta_0,\delta_0][x_1;\theta_1,\delta_1] \cdots [x_n;\theta_n,\delta_n]_{(\lambda_0,\lambda_1,...,\lambda_n)}, \]
the subalgebra
\[ A:=\mathbf{k}[x_1;\theta_1,\delta_1] \cdots [x_n;\theta_n,\delta_n]_{(\lambda_1,...,\lambda_n)}\]
is also a symmetric integral Poisson-CGL extension. Let $p$ be the map defined in Theorem \ref{HPP-Poi-CGL} for $A$. Note that $B$ can be viewed as a Poisson-Cauchon extension (see Definition \ref{T-P-O}) $B=A[x_0;\theta,\delta]_{\lambda_0}$ of $A$, where
\begin{equation}
\label{theta,delta}
\theta(x_j)x_0=-\theta_j(x_0)x_j \ \text{and} \ \delta(x_j)=-\delta_j(x_0), \ \ j \in [1,n].
\end{equation}
Let $\eta$ be as in 2) of Definition \ref{T-P-O} for $A[x_0;\theta,\delta]_{\lambda_0}$. Assume $A$ has a preferred quantization
\[ \mathcal{A}=L[X_1;\sigma_1,\Delta_1]\cdots [X_n;\sigma_n,\Delta_n]_{(\lambda_1,...,\lambda_n)}. \]
Let $Y_{\mathcal{A}^{ex}}=(Y_1,...,Y_n)$ be the sequence of homogeneous prime elements of $\mathcal{A}^{ex}$. By Lemma \ref{log-can-qy} and by 3) of Definition \ref{Q-CGL}, for $i,j \in [1,n]$,
\[ Y_i \ast Y_j-q^{l_{i,j}}Y_j \ast Y_i \ \ \text{for some} \ \ l_{i,j} \in \mathbb{Z}. \]
Let
\[ \Gamma_q=\frac{L<Y^{\pm 1}_1,...,Y^{\pm 1}_n>}{(Y_i \ast Y_j-q^{l_{i,j}}Y_j \ast Y_i)}. \]
Notice that $\Gamma_q$ is the $L$-form of the quantum torus of $\mathcal{A}^{ex}$. By Theorem \ref{HP-CGL}, for $j \in [1,n]$, $Y_j=Y_{p(j)}X_j-C_j$, where $C_j \in \mathcal{A}^{ex}_{j-1}$. Since $Y_j$ and $Y_{p(j)}$ lie in $\mathcal{A}$, $C_j \in \mathcal{A}$, from which one sees that $C_j \in \mathcal{A}^{ex}_{j-1} \cap \mathcal{A}=\mathcal{A}_{j-1}$. It follows that $\mathcal{A} \subset \Gamma_q$.

Our goal is to construct a quantum-Ore extension $\mathcal{B}=\mathcal{A}[X_0;\sigma,\Delta]$ of $\mathcal{A}$ such that it can be written as a symmetric quantum-CGL extension
\[ \mathcal{B}=L[X_0;\sigma_0,\Delta_0][X_1;\sigma_1,\Delta_1]\cdots [X_n;\sigma_n,\Delta_n]_{(\lambda_0,\lambda_1,...,\lambda_n)} \]
that is a preferred quantization of $B$.

In order to construct $\mathcal{B}$, we need to choose $\sigma$ and $\Delta$ properly. The choice of $\sigma$ is fixed by $\theta$. Indeed, since $B$ is integral, $\theta(x_j)=l_jx_j$ for some integer $l_j$, $j \in [1,n]$. Then $\sigma$ has to be given by
\begin{equation}
\label{sigma}
\sigma(X_j)=q^{l_{j}}X_j, \ \ j \in [1,n].
\end{equation}
To construct $\Delta$, we first construct an element $D \in \Gamma_q$ and define $\Delta : \Gamma_q \rightarrow \Gamma_q$ by
\begin{equation}
\label{Delta}
\Delta(r)=D \ast r-\sigma(r) \ast D, \ \ r \in \Gamma_q,
\end{equation}
which has the property that
\[ \Delta(\mathcal{A}) \subset \mathcal{A}. \]
It is straightforward to check that $\Delta$ is a $\sigma$-derivation.

We now describe the construction of $D \in \Gamma_q$. When $\delta=0$, we just let $D=0$, which means $\Delta=0$. Now assume $\delta \neq 0$. Let $y_A=(y_1,...,y_n)$ be the sequence of homogeneous Poisson prime elements of $A$ and let
\[ \Gamma=\mathbf{k}[y^{\pm 1}_1,...,y^{\pm 1}_n] \]
be its Poisson torus. Define the $\mathbf{k}$-linear map $f : \Gamma \rightarrow \Gamma_q$ by
\[ f(y^v)=Y^v, \ \ v=(v_1,...,v_n) \in \mathbb{Z}^n, \]
where $y^v=y^{v_1}_1 \cdots y^{v_n}_n$ and $Y^v=Y^{v_1}_1 \ast \cdots \ast Y^{v_n}_n$. Consider now the presentation
\[ B=\mathbf{k}[x_1;\theta_1,\delta_1] \cdots [x_n;\theta_n,\delta_n][x_0;\theta,\delta]_{(\lambda_1,...,\lambda_n,\lambda_0)}, \]
as a Poisson-CGL extension, where $\theta$ and $\delta$ are defined in (\ref{theta,delta}). As $\delta \neq 0$, by Theorem \ref{HPP-Poi-CGL}, there exists a unique $k \in [1,n]$ such that $\delta(y_k) \neq 0$ and $y_k$ is a homogeneous Poisson prime element of $A$. Let $m$ be the largest nonnegative integer such that $p^{m}(k) \in [1,n]$. For $i \in [0,m-1]$, let $B_{p^{i}(k)}$ be as in Lemma-Notation \ref{qmon}. It is easy to see that those $B_{p^{i}(k)}$ lie in $\Gamma_q$. For $j \in [1,n]$, let $\eta_j$ be as in 3) of Definition \ref{Q-CGL} for $\mathcal{A}$ and let $\omega_j=q^{\eta_j}$ (note that those $\omega_j$ are exactly those in 3) of Definition \ref{T-CGL} for $\mathcal{A}^{ex}$). We define $D$ as follows:
\[
D=\left( \sum^m_{i=0} (\omega_{p^{i}(k)}Y^{-1}_{p^{i}(k)} \ast B_{p^{i}(k)}) \ast (\omega_{p^{i+1}(k)}Y^{-1}_{p^{i+1}(k)} \ast B_{p^{i+1}(k)}) \ast \cdots \ast \right.
\]
\begin{equation}
\label{Dis}
\left. (\omega_{p^{m-1}(k)}Y^{-1}_{p^{m-1}(k)} \ast B_{p^{m-1}(k)}) \right) \ast f\left(\frac{\delta(x_{p^{m}(k)})}{\eta x_{p^{m}(k)}}\right) \in \Gamma_q.
\end{equation}
We have the following two theorems, which are proven in $\S$\ref{subsec-proofs}.

\begin{theorem}
\label{one-ste-qua}
Given a symmetric integral Poisson-CGL extension
\[ B=\mathbf{k}[x_0;\theta_0,\delta_0][x_1;\theta_1,\delta_1] \cdots [x_n;\theta_n,\delta_n]_{(\lambda_0,\lambda_1,...,\lambda_n)}, \]
assume
\[ \mathcal{A}=L[X_1;\sigma_1,\Delta_1]\cdots [X_n;\sigma_n,\Delta_n]_{(\lambda_1,...,\lambda_n)} \]
is a preferred quantization of
\[ A:=\mathbf{k}[x_1;\theta_1,\delta_1] \cdots [x_n;\theta_n,\delta_n]_{(\lambda_1,...,\lambda_n)}. \]
Let $\sigma$ be defined by (\ref{sigma}). When $\delta=0$, let $D=0$. When $\delta \neq 0$, let $D \in \Gamma_q$ be defined by (\ref{Dis}). Let $\Delta$ be defined by $(\ref{Delta})$. Then $\Delta(\mathcal{A}) \subset \mathcal{A}$ and $\mathcal{B}=\mathcal{A}[X_0;\sigma,\Delta]$ is a quantum-Ore extension of $\mathcal{A}$ such that $\mathcal{B}$ can be written as a symmetric quantum-CGL extension
\[ \mathcal{B}=L[X_0;\sigma_0,\Delta_0][X_1;\sigma_1,\Delta_1]\cdots [X_n;\sigma_n,\Delta_n]_{(\lambda_0,\lambda_1,...,\lambda_n)} \]
that is a preferred quantization of $B$.
\end{theorem}

Preferred quantization is unique in the following sense:

\begin{theorem}
\label{unique}
Let $\mathcal{B}'=\mathcal{A}[X'_0;\sigma,\Delta']$ be a quantum-Ore extension of $\mathcal{A}$. If it can be written as a symmetric quantum-CGL extension that is a preferred quantization of
\[ B=\mathbf{k}[x_0;\theta_0,\delta_0][x_1;\theta_1,\delta_1] \cdots [x_n;\theta_n,\delta_n]_{(\lambda_0,\lambda_1,...,\lambda_n)}, \]
then
\[ \Delta'=\epsilon \Delta, \ \text{for some} \ \epsilon \in L \ \ \text{satisfying} \ \ \epsilon|_{q=1}=1. \]
\end{theorem}

By induction and by Theorem \ref{one-ste-qua}, we have the main theorem of this paper:

\begin{theorem}
\label{pre-quantization}
Every symmetric integral Poisson-CGL extension has a preferred quantization.
\end{theorem}

\subsection{Proofs of Theorems \ref{one-ste-qua} and \ref{unique}} \label{subsec-proofs}

Identify $A$ with the semiclassical limit of $\mathcal{A}$ and we have the following lemma:

\begin{lemma}
\label{y-Y}
For $j \in [1,n]$, $y_j=Y_j+(q-1)\mathcal{A}$.
\end{lemma}

\begin{proof}
Clearly, $y_1=x_1=X_1+(q-1)\mathcal{A}=Y_1+(q-1)\mathcal{A}$. Assume Lemma \ref{y-Y} holds for $y_1,...,y_{l-1}$ for some $l \in [2,n]$. It is straightforward to see that $Y_l+(q-1)\mathcal{A}$ is a homogeneous Poisson element of $A_l$. Then $Y_l+(q-1)\mathcal{A}$ is either a constant or a product of homogeneous Poisson prime elements of $A_l$ by Lemma \ref{prime factor}. By the definition of $Y_l$ and by the inductive assumption,
\[ (Y_l+(q-1)\mathcal{A})-y_{p(l)}x_l \in A_{l-1}. \]
It follows that $Y_l+(q-1)\mathcal{A}$ must have a prime factor that is a degree one polynomial of $x_l$ with coefficients in $A_{l-1}$. By Theorem \ref{HPP-Poi-CGL}, that prime factor has to be a scalar multiple of $y_l$. Then $Y_l+(q-1)\mathcal{A}=ay_l$ for some $a \in A_{l-1}$. Since $y_l-y_{p(l)}x_l \in A_{l-1}$, $ay_l-y_l \in A_{l-1}$. It follows that $a=1$ and $y_l=Y_l+(q-1)\mathcal{A}$. By induction, $y_j=Y_j+(q-1)\mathcal{A}$ for $j \in [1,n]$.
\end{proof}

Consider the case when $\delta \neq 0$. Recall that there exists a unique $k \in [1,n]$ such that $\delta(y_k) \neq 0$ and $y_k$ is a homogeneous Poisson prime element of $A$. Let $m$ be the largest nonnegative integer such that $p^{m}(k) \in [1,n]$. Recall from $\S$\ref{subsec-sym-Poi-CGL} that, for $j \in [1,n]$, $A_{j}[x_0;\theta,\delta]_{\lambda_0}$ is a Poisson-Cauchon of $A_j$ and
\[ \mathbf{k}[x_1;\theta_1,\delta_1] \cdots [x_j;\theta_j,\delta_j][x_0;\theta,\delta]_{(\lambda_1,...,\lambda_j,\lambda_0)} \]
is a Poisson-CGL extension. Here $\theta$ and $\delta$ are the restrictions of $\theta$ and $\delta$ in (\ref{theta,delta}) to $A_j$. For $i \in [0,m]$, let $d^{(i)}$ be the distinguished element with respect to $A_{p^{i-1}(k)-1}[x_0;\theta,\delta]_{\lambda_0}$ ($p^{-1}(k):=n+1$) and let $d^{(m+1)}=0$. By 1) of Lemma \ref{d-M} and by 1) of Lemma \ref{one-one}, $d^{(i)}-d^{(i+1)}$, $d^{(i+1)}-d^{(i+2)}$,..., $d^{(m)}-d^{(m+1)}$ are the monomial terms of $d^{(i)} \in \Gamma$ for $i \in [0,m]$. In particular, $d^{(m)}$ is a monomial in $\Gamma$. Let $b_{p^{i}(k)}$ be as in Lemma-Notation \ref{mon}.

\begin{lemma}
\label{d-y}
Let $i \in [0,m]$. For $l \in [1,p^{i-1}(k)-1]$,
\begin{align*}
&\{d^{(i)}-d^{(i+1)}, \ y_l\}=\theta(y_l)(d^{(i)}-d^{(i+1)}), \ \ l \neq p^{i}(k), \ \text{and} \\
&\{d^{(i)}-d^{(i+1)}, \ y_{p^{i}(k)}\}=(\theta(y_{p^{i}(k)})+\eta y_{p^i(k)})(d^{(i)}-d^{(i+1)}).
\end{align*}
\end{lemma}

\begin{proof}
Recall from Lemma-Notation \ref{mon} that
\[ \Gamma_{p^{i-1}(k)-1}=\mathbf{k}[y^{\pm 1}_1,...,y^{\pm 1}_{p^{i-1}(k)-1}], \ \Gamma'_{p^{i-1}(k)-1}=\mathbf{k}[y^{\pm 1}_1,...,y^{\pm 1}_{p^{i}(k)-1},y^{\pm 1}_{p^{i}(k)+1},...,y^{\pm 1}_{p^{i-1}(k)-1}]. \]
For $l \in [1,p^{i-1}(k)-1]$, by (\ref{dis1}),
\[ \{d^{(i)}, \ y_l\}-\theta(y_l)d^{(i)}=\delta(y_l) \in A_{p^{i-1}(k)-1}. \]
Since $y_{p^{i}(k)}$ is a homogeneous Poisson prime element of $A_{p^{i-1}(k)-1}$, by Theorem \ref{HPP-Poi-CGL}, $A_{p^{i-1}(k)-1} \subset \Gamma'_{p^{i-1}(k)-1}[y_{p^{i}(k)}]$. Thus,
\[ \{d^{(i)}, \ y_l\}-\theta(y_l)d^{(i)} \in \Gamma'_{p^{i-1}(k)-1}[y_{p^{i}(k)}], \ \ l \in [1,p^{i-1}(k)-1]. \]
As $d^{(i+1)} \in \Gamma'_{p^{i-1}(k)-1}[y_{p^{i}(k)}]$,
\[ \{d^{(i)}-d^{(i+1)}, \ y_l\}-\theta(y_l)(d^{(i)}-d^{(i+1)}) \in \Gamma'_{p^{i-1}(k)-1}[y_{p^{i}(k)}], \ \ l \in [1,p^{i-1}(k)-1]. \]
By 1) of Lemma \ref{d-M}, $d^{(i)}-d^{(i+1)}$ is a monomial in $\Gamma'_{p^{i-1}(k)-1}[y^{\pm 1}_{p^{i}(k)}]$ with degree -1. Then
\[ \{d^{(i)}-d^{(i+1)}, \ y_l\}=\theta(y_l)(d^{(i)}-d^{(i+1)}), \ \ l \in [1,p^{i-1}(k)-1] \setminus \{p^i(k)\}. \]
Through direct computation, $\{d^{(i)}, \ y_{p^i(k)}\}=(\theta(y_{p^{i}(k)})+\eta y_{p^i(k)})d^{(i)}$. As $d^{(i)}-d^{(i+1)}$ is a monomial term of $d^{(i)} \in \Gamma$,
\[ \{d^{(i)}-d^{(i+1)}, \ y_{p^{i}(k)}\}=(\theta(y_{p^{i}(k)})+\eta y_{p^i(k)})(d^{(i)}-d^{(i+1)}).\]
\end{proof}

Define $D^{(m+1)}=0$ and $D^{(m)}=f\left(\frac{\delta(x_{p^{m}(k)})}{\eta x_{p^{m}(k)}}\right) \in \Gamma_q$. Inductively define
\begin{equation}
\label{D-D}
D^{(i)}=D^{(i+1)}+\omega_{p^{i}(k)}Y^{-1}_{p^{i}(k)} \ast B_{p^{i}(k)} \ast (D^{(i+1)}-D^{(i+2)}) \in \Gamma_q, \ \ i \in [0,m-1].
\end{equation}
It is easy to check that $D=D^{(0)}$. Recall from $\S$\ref{subsec-sym-Poi-CGL} that $d^{(m)}=\frac{\delta(y_{p^{m}(k)})}{\eta y_{p^{m}(k)}}=\frac{\delta(x_{p^{m}(k)})}{\eta x_{p^{m}(k)}}$. Hence, $D^{(m)}=f(d^{(m)})$ is a monomial in $\Gamma_q$. Inductively, one sees that $D^{(i)}-D^{(i+1)}$, $D^{(i+1)}-D^{(i+2)}$,..., $D^{(m)}-D^{(m+1)}$ are the monomial terms of $D^{(i)} \in \Gamma_q$ for $i \in [0,m]$.

\begin{lemma}
\label{lie in}
For $i \in [0,m]$,

1) $d^{(i)}-d^{(i+1)}=D^{(i)}-D^{(i+1)}+(q-1)\Gamma_q$;

2) $Y_{p^{i}(k)} \ast D^{(i)} \in \mathcal{A}_{p^{i}(k)}$.
\end{lemma}

\begin{proof}
Keep in mind that, for $j \in [1,n]$, $y_j=Y_j + (q-1)\Gamma_q$ by Lemma \ref{y-Y}.

1) For $j \in [1,n]$, it follows from $y_j=Y_j + (q-1)\Gamma_q$ that $c_j=C_j + (q-1)\Gamma_q$. By Lemma-Notation \ref{mon} and Lemma-Notation \ref{qmon}, $b_{p^{i}(k)}=B_{p^{i}(k)}+(q-1)\Gamma_q$ for $i \in [0,m]$. Since $D^{(m)}=f(d^{(m)})$, $d^{(m)}-d^{(m+1)}=D^{(m)}-D^{(m+1)}+(q-1)\Gamma_q$. By definition, $\omega_{p^{i}(k)}+(q-1)\Gamma_q=1$ for $i \in [0,m]$. By 2) of Lemma \ref{d-M} and by (\ref{D-D}), one inductively sees that
\[ d^{(i)}-d^{(i+1)}=D^{(i)}-D^{(i+1)}+(q-1)\Gamma_q, \ \ i \in [0,m]. \]

2) As $d^{(m)}=\frac{\delta(y_{p^{m}(k)})}{\eta y_{p^{m}(k)}}$ is a monomial in $\Gamma$ and $D^{(m)}=f\left(\frac{\delta(y_{p^{m}(k)})}{\eta y_{p^{m}(k)}}\right)$, $Y_{p^{m}(k)} \ast D^{(m)}$ is a scalar multiple of $f(\delta(y_{p^m(k)}))$. Since $\delta(y_{p^{m}(k)}) \in A_{p^{m}(k)-1}$, $\delta(y_{p^{m}(k)})$ is a scalar multiple of $y^v$ for some $v \in \mathbb{N}^n$ and $y^v \in A_{p^{m}(k)-1}$. Thus, $f(\delta(y_{p^{m}(k)})) \in \mathcal{A}_{p^{m}(k)-1}$, which implies that
\[ Y_{p^{m}(k)} \ast D^{(m)} \in \mathcal{A}_{p^{m}(k)-1} \subset \mathcal{A}_{p^{m}(k)}. \]
Assume 2) of Lemma \ref{lie in} holds for $D^{(m)}$,$D^{(m-1)}$,...,$D^{(i)}$ and consider $D^{(i-1)}$. It follows from the definition that
\[ Y_{p^{i-1}(k)} \ast D^{(i-1)}=Y_{p^{i-1}(k)} \ast D^{(i)}+\omega_{p^{i-1}(k)}B_{p^{i-1}(k)} \ast (D^{(i)}-D^{(i+1)}).  \]
Through direct computation, one has
\[ Y_{p^{i-1}(k)} \ast D^{(i)}=X_{p^{i-1}(k)} \ast \sigma^{-1}_{p^{i-1}(k)}(Y_{p^{i}(k)}) \ast D^{(i)}-\omega_{p^{i-1}(k)}C_{p^{i-1}(k)} \ast D^{(i)}. \]
As $Y_{p^{i}(k)} \ast D^{(i)} \in \mathcal{A}_{p^{i}(k)}$, $Y_{p^{i-1}(k)} \ast D^{(i-1)} \in \mathcal{A}_{p^{i-1}(k)}$ if and only if
\[ B_{p^{i-1}(k)} \ast (D^{(i)}-D^{(i+1)})-C_{p^{i-1}(k)} \ast D^{(i)} \in \mathcal{A}_{p^{i-1}(k)}. \]
By Lemma \ref{d-M}, $M^{(i-1)}:=y_{p^{i}(k)}b_{p^{i-1}(k)}(d^{(i)}-d^{(i+1)}) \in A_{p^{i-1}(k)-1}$. As $M^{(i-1)}$ is a monomial in $\Gamma$, $M^{(i-1)}$ is a scalar multiple of $y^v$ for some $v \in \mathbb{N}^n$ and $y^v \in A_{p^{i-1}(k)-1}$. By part 1) and by the fact that $b_{p^{i}(k)}=B_{p^{i}(k)}+(q-1)\Gamma_q$,
\[ Y_{p^{i}(k)} \ast B_{p^{i-1}(k)} \ast (D^{(i)}-D^{(i+1)})+(q-1)\Gamma_q=M^{(i-1)}. \]
Since $Y_{p^{i}(k)} \ast B_{p^{i-1}(k)} \ast (D^{(i)}-D^{(i+1)})$ is a monomial in $\Gamma_q$, it is a scalar multiple of $f(M^{(i-1)}) \in \mathcal{A}_{p^{i-1}(k)-1}$. Hence, $Y_{p^{i}(k)} \ast B_{p^{i-1}(k)} \ast (D^{(i)}-D^{(i+1)}) \in \mathcal{A}_{p^{i-1}(k)-1}$. By the inductive assumption, as $Y_{p^{i}(k)}$ is prime, hence normal, in $\mathcal{A}_{p^{i-1}(k)-1}$, one has $Y_{p^{i}(k)} \ast C_{p^{i-1}(k)} \ast D^{(i)} \in \mathcal{A}_{p^{i-1}(k)-1}$. Hence,
\begin{equation}
\label{in1}
Y_{p^{i}(k)} \ast (B_{p^{i-1}(k)} \ast (D^{(i)}-D^{(i+1)})-C_{p^{i-1}(k)} \ast D^{(i)}) \in \mathcal{A}_{p^{i-1}(k)-1}.
\end{equation}
It is easy to check that $B_{p^{i-1}(k)} \ast (D^{(i)}-D^{(i+1)})-C_{p^{i-1}(k)} \ast D^{(i)}$ contains no negative power of $Y_{p^{i}(k)}$. It follows that
\[ B_{p^{i-1}(k)} \ast (D^{(i)}-D^{(i+1)})-C_{p^{i-1}(k)} \ast D^{(i)}=a_1 \ast a^{-1}_2, \]
for some $a_1,a_2 \in \mathcal{A}_{p^{i-1}(k)-1}$ and $a_2 \notin Y_{p^{i}(k)} \ast \mathcal{A}_{p^{i-1}(k)-1}$. By (\ref{in1}),
\[ Y_{p^{i}(k)} \ast (B_{p^{i-1}(k)} \ast (D^{(i)}-D^{(i+1)})-C_{p^{i-1}(k)} \ast D^{(i)}) \ast a_2 \in Y_{p^{i}(k)} \ast \mathcal{A}_{p^{i-1}(k)-1}. \]
Since $Y_{p^{i}(k)}$ is prime in $\mathcal{A}_{p^{i-1}(k)-1}$ by Lemma \ref{prime} and $a_2 \notin Y_{p^{i}(k)} \ast \mathcal{A}_{p^{i-1}(k)-1}$, one has
\[ Y_{p^{i}(k)} \ast (B_{p^{i-1}(k)} \ast (D^{(i)}-D^{(i+1)})-C_{p^{i-1}(k)} \ast D^{(i)}) \in Y_{p^{i}(k)} \ast \mathcal{A}_{p^{i-1}(k)-1}, \]
which implies that
\[ B_{p^{i-1}(k)} \ast (D^{(i)}-D^{(i+1)})-C_{p^{i-1}(k)} \ast D^{(i)} \in \mathcal{A}_{p^{i-1}(k)-1} \subset \mathcal{A}_{p^{i-1}(k)}. \]
Therefore, 2) of Lemma \ref{lie in} is satisfied for $D^{(i-1)}$. By induction, 2) of Lemma \ref{lie in} has been proven.
\end{proof}

{\bf Proof of Theorem \ref{one-ste-qua}}: We first need to prove that $\Delta(\mathcal{A}) \subset \mathcal{A}$. When $\delta=0$, $D=0$ and there is nothing to prove. 

Assume $\delta \neq 0$. Let $\Gamma'=\mathbf{k}[y^{\pm 1}_1,...,y^{\pm 1}_{k-1},y^{\pm 1}_{k+1},...,y^{\pm 1}_{n}]$. By Lemma \ref{d-y},
\begin{equation}
\label{in3}
\{d^{(0)}-d^{(1)}, \ y_j\}-\theta(y_j)(d^{(0)}-d^{(1)}) \in \Gamma'[y_k], \ \ j \in [1,n].
\end{equation}
Let $\Gamma'_q$ be the quantum $L$-torus generated by $Y_1,...,Y_{k-1},Y_{k+1},...,Y_{n}$. Since $Y_k$ is a homogeneous prime element of $\mathcal{A}^{ex}$, by Theorem \ref{HP-CGL}, $\mathcal{A} \subset \Gamma'_q[Y_k]$. By definition, $D=D^{(0)}=(D^{(0)}-D^{(1)})+D^{(1)}$, where $D^{(0)}-D^{(1)}$ is a monomial in $\Gamma_q$ and $D^{(1)} \in \Gamma'_q$. By 1) of Lemma \ref{lie in}, $D^{(0)}-D^{(1)}+(q-1)\Gamma_q=d^{(0)}-d^{(1)}$. By (\ref{in3}),
\[ (D^{(0)}-D^{(1)}) \ast Y_j-\sigma(Y_j) \ast (D^{(0)}-D^{(1)}) \in \Gamma'_q[Y_k], \ \ j \in [1,n]. \]
As $D^{(1)} \in \Gamma'_q$ and $D=D^{(0)}$,
\[ D \ast Y_j-\sigma(Y_j) \ast D \in \Gamma'_q[Y_k], \ \ j \in [1,n]. \]
Let $a$ be an arbitrary element in $\mathcal{A}$. As $a \in \Gamma'_q[Y_k]$,
\[ D \ast a-\sigma(a) \ast D \in \Gamma'_q[Y_k],\]
which implies that
\[ D \ast a-\sigma(a) \ast D=a_1 \ast a^{-1}_2, \]
for some $a_1,a_2 \in \mathcal{A}$ and $a_2 \notin Y_{k} \ast \mathcal{A}$. By part 2) of Lemma \ref{lie in}, $Y_k \ast D \in \mathcal{A}$. Since $Y_{k}$ is prime, hence normal, in $\mathcal{A}$ by Lemma \ref{prime},
\begin{equation}
\label{in2}
Y_k \ast (D \ast a-\sigma(a) \ast D) \in \mathcal{A}.
\end{equation}
It follows that
\[ Y_{k} \ast (D \ast a-\sigma(a) \ast D) \ast a_2 \in Y_{k} \ast \mathcal{A}. \]
Since $Y_{k}$ is prime in $\mathcal{A}$ and $a_2 \notin Y_{k} \ast \mathcal{A}$, one has
\[ Y_{k} \ast (D \ast a-\sigma(a) \ast D) \in Y_{k} \ast \mathcal{A}, \]
which implies that
\[ D \ast a-\sigma(a) \ast D \in \mathcal{A}. \]
Therefore, $\Delta(\mathcal{A}) \subset \mathcal{A}$.

We have proved that $\mathcal{B}=\mathcal{A}[X_0;\sigma,\Delta]$ is an Ore extension of $\mathcal{A}$, where $\sigma$ and $\Delta$ are $L$-linear maps. By the definition of $\sigma$ and $\Delta$, it is easy to check that $\mathcal{B}/(q-1)\mathcal{B}$ is commutative. It is clear that $\mathcal{B}$ is a free $L$-module. Hence, $\mathcal{B}$ is a quantum $L$-algebra and $\mathcal{B}=\mathcal{A}[X_0;\sigma,\Delta]$ is a quantum-Ore extension of $\mathcal{A}$. We still need to show two things:

a) $\mathcal{B}$ can be written as a symmetric quantum-CGL extension
\[ \mathcal{B}=L[X_0;\sigma_0,\Delta_0][X_1;\sigma_1,\Delta_1]\cdots [X_n;\sigma_n,\Delta_n]_{(\lambda_0,\lambda_1,...,\lambda_n)}; \]

b) $\mathcal{B}$ is a preferred quantization of $B$.\\
{\bf To prove a)}: When $\delta=0$, $\Delta=0$ and there is nothing to show.

Assume $\delta \neq 0$. For $j \in [1,n]$, let $\Gamma_j=\mathbf{k}[y^{\pm 1}_1,...,y^{\pm 1}_{j}]$ and let $\Gamma_{q,j}$ be the quantum $L$-torus generated by $Y_1,...,Y_j$. Let $l \in [1,n]$. We choose $i \in [0,m+1]$ such that $p^{i}(k) \leq l < p^{i-1}(k)$. Then
\begin{align}
\label{deltal}
&\delta(x_l)=\{d^{(i)}, \ x_l\}-\theta(x_l)d^{(i)}, \ \ \text{and} \\
\label{Deltal}
&\Delta(X_l)=D^{(i)} \ast X_l-\sigma(X_l) \ast D^{(i)}.
\end{align}
By Theorem \ref{HPP-Poi-CGL}, $x_l=\frac{y_l+c_l}{y_{p(l)}}$, where $c_l \in A_{l-1}$. By (\ref{deltal}),
\[ \delta(x_l)=\{d^{(i)}, \ x_l\}-\theta(x_l)d^{(i)}=\left\{d^{(i)}, \ \frac{y_l+c_l}{y_{p(l)}}\right\}-\theta\left(\frac{y_l+c_l}{y_{p(l)}}\right)d^{(i)}. \]
Since $B$ is symmetric, $\delta(x_l) \in A_{l-1} \subset \Gamma_{l-1}$, which means
\begin{equation}
\label{in4}
\left\{d^{(i)}, \ \frac{y_l+c_l}{y_{p(l)}}\right\}-\theta\left(\frac{y_l+c_l}{y_{p(l)}}\right)d^{(i)} \in \Gamma_{l-1}.
\end{equation}

Case 1: $p^{i}(k) < l$. In this case, $d^{(i)} \in \Gamma_{l-1}$ and $\frac{c_l}{y_{p(l)}} \in \Gamma_{l-1}$. By (\ref{in4}),
\begin{equation}
\label{in6}
\left\{d^{(i)}, \ \frac{y_l}{y_{p(l)}}\right\}=\theta\left(\frac{y_l}{y_{p(l)}}\right)d^{(i)}.
\end{equation}
As $d^{(i)}-d^{(i+1)}$,$d^{(i+1)}-d^{(i+2)}$,...,$d^{(m)}-d^{(m+1)}$ are the monomial terms of $d^{(i)} \in \Gamma$, they all satisfy (\ref{in6}) by replacing $d^{(i)}$. As $D^{(i)}-D^{(i+1)}$,$D^{(i+1)}-D^{(i+2)}$,...,$D^{(m)}-D^{(m+1)}$ are the monomial terms of $D^{(i)} \in \Gamma_q$, by 1) of Lemma \ref{lie in},
\[ D^{(i)} \ast (Y^{-1}_{p(l)} \ast Y_l)=\sigma(Y^{-1}_{p(l)} \ast Y_l) \ast D^{(i)}. \]
By Theorem \ref{HP-CGL}, $X_l=Y^{-1}_{p(l)} \ast (Y_l+C_l)$, where $C_l \in \mathcal{A}_{l-1}$. It follows that $Y^{-1}_{p(l)} \ast C_l \in \Gamma_{q,l-1}$. Since $D^{(i)} \in \Gamma_{q,l-1}$, by (\ref{Deltal}), $\Delta(X_l) \in \Gamma_{q,j-1}$.

Case 2: $p^{i}(k) = l$. In this case, $i \neq m+1$, otherwise $l=0$. By Lemma \ref{d-y},
\[ \{d^{(i)}-d^{(i+1)}, \ r\} =\theta(r) (d^{(i)}-d^{(i+1)}), \ \ r \in \Gamma_{p^i(k)-1}. \]
As $D^{(i)}-D^{(i+1)}$ is a monomial in $\Gamma_q$, by 1) of Lemma \ref{lie in},
\[ (D^{(i)}-D^{(i+1)}) \ast \mathbf{r}=\sigma(\mathbf{r}) \ast (D^{(i)}-D^{(i+1)}), \ \ \mathbf{r} \in \Gamma_{q,p^i(k)-1}. \]
As $Y^{-1}_{p^{i+1}(k)} \ast C_{p^{i}(k)} \in \Gamma_{q,p^i(k)-1}$,
\[ (D^{(i)}-D^{(i+1)}) \ast (Y^{-1}_{p^{i+1}(k)} \ast C_{p^{i}(k)})=\sigma(Y^{-1}_{p^{i+1}(k)} \ast C_{p^{i}(k)}) \ast (D^{(i)}-D^{(i+1)}). \]
It follows from $D^{(i+1)} \in \Gamma_{q,p^{i}(k)-1}$ that
\[ D^{(i)} \ast (Y^{-1}_{p^{i+1}(k)} \ast C_{p^{i}(k)})-\sigma(Y^{-1}_{p^{i+1}(k)} \ast C_{p^{i}(k)}) \ast D^{(i)} \in \Gamma_{q,p^{i}(k)-1}. \]
By (\ref{in4}),
\begin{equation}
\label{in5}
\left\{d^{(i)}, \ \frac{y_{p^{i}(k)}}{y_{p^{i+1}(k)}}\right\}-\theta\left(\frac{y_{p^{i}(k)}}{y_{p^{i+1}(k)}}\right)d^{(i)} \in \Gamma_{l-1}.
\end{equation}
As $d^{(i)}-d^{(i+1)}$,$d^{(i+1)}-d^{(i+2)}$,...,$d^{(m)}-d^{(m+1)}$ are the monomial terms of $d^{(i)} \in \Gamma$, they all satisfy (\ref{in5}) by replacing $d^{(i)}$. As $D^{(i)}-D^{(i+1)}$,$D^{(i+1)}-D^{(i+2)}$,...,$D^{(m)}-D^{(m+1)}$ are the monomial terms of $D^{(i)} \in \Gamma_q$, by 1) of Lemma \ref{lie in},
\[ D^{(i)} \ast (Y^{-1}_{p^{i+1}(k)} \ast Y_{p^{i}(k)})-\sigma(Y^{-1}_{p^{i+1}(k)} \ast Y_{p^{i}(k)}) \ast D^{(i)} \in \Gamma_{q,l-1}. \]
By (\ref{Deltal}), $\Delta(X_l) \in \Gamma_{q,j-1}$.

Therefore, in both cases, $\Delta(X_l) \in \Gamma_{q,j-1}$. We have proved that $\Delta(\mathcal{A}) \subset \mathcal{A}$, which implies
\[ \Delta(X_j) \in (\mathcal{A} \cap \Gamma_{q,j-1})=\mathcal{A}_{j-1}. \]

Given that $\Delta(X_j) \in \mathcal{A}_{j-1}, \ j \in [1,n]$, it follows directly from the fact that
\[ B=\mathbf{k}[x_0;\theta_0,\delta_0][x_1;\theta_1,\delta_1] \cdots [x_n;\theta_n,\delta_n]_{(\lambda_0,\lambda_1,...,\lambda_n)} \]
is a symmetric integral Poisson-CGL extension that
\[ \mathcal{B}=L[X_0;\sigma_0,\Delta_0][X_1;\sigma_1,\Delta_1]\cdots [X_n;\sigma_n,\Delta_n]_{(\lambda_0,\lambda_1,...,\lambda_n)} \]
is a symmetric quantum-CGL extension.\\
{\bf To prove b)}: Note that by 1) of Lemma \ref{d-M}, $d=D+(q-1)\Gamma_q$. By the definition of $\sigma$ and $\Delta$, it is easy to see that $\mathcal{B}$ is a quantization of $B$ in the sense of Definition \ref{ite-qua}.

By Remark \ref{extension}, $\mathcal{B}^{ex}$ is a symmetric CGL extension over $\mathbb{K}$. Define the $\mathbb{N}^{n+1}$-valued degree for elements in $B$ and $\mathcal{B}^{ex}$ as in $\S$\ref{subsec-not}. Denote by $\mathcal{B}^{ex}_{[0,j]}$ the subalgebra of $\mathcal{B}^{ex}$ generated by $X_0,X_1,...,X_j$, for $j \in [0,n]$.  Let $Y_{\mathcal{B}^{ex}}=(Y'_0,Y'_1,...,Y'_n)$ be the sequence of homogeneous prime elements of $\mathcal{B}^{ex}$. By definition, $Y'_0=X_0 \in \mathcal{B}$. Let $Y_B=(y'_0,y'_1,...,y'_n)$ be the sequence of homogeneous Poisson prime elements of $B$. By definition, $y'_0=x_0$.

Consider the case when $\delta=0$, which implies $\Delta=0$. It is easy to see that $y'_j=y_j$ and $Y'_j=Y_j$ for $j \in [1,n]$. Since $\mathcal{A}^{ex}$ and $\mathcal{A}$ share the same level sets, by 2) of Lemma \ref{one-one} and by 2) of Lemma \ref{qone-one}, $y_j$ and $Y_j$ have the same degree for $j \in [1,n]$. Clearly, $Y'_0=X_0$ and $y'_0=x_0$ have the same degree. Hence, $y'_j$ and $Y'_j$ have the same degree for $j \in [0,n]$. By 2) of Lemma \ref{one-one} and by 2) of Lemma \ref{qone-one}, $B$ and $\mathcal{B}^{ex}$ share the same level sets.

Consider the case when $\delta \neq 0$. As $\mathcal{B}^{ex}$ is symmetric, $\mathcal{B}^{ex}_{[0,j]}$ has the presentation
\[ \mathbb{K}[X_1;\sigma_1,\Delta_1]\cdots [X_j;\sigma_n,\Delta_n][X_0;\sigma,\Delta]_{(\lambda_1,...,\lambda_n,\lambda_0)} \]
as a CGL extension, for $j \in [1,n]$. By (\ref{Deltal}) and by the definition of $D^{(m)}$, one sees that $\Delta(Y_{p^m(k)})=\Delta(X_{p^m(k)}) \neq 0$. By Lemma \ref{qk-p(k)}, for $j \in [1,n]$, $\Delta(Y_j) \neq 0$ if and only if $j=p^i(k)$ for some $i \in [0,m]$.

Case 1: $j=p^i(k)$ for some $i \in [0,m]$. By Theorem \ref{HP-CGL},
\begin{equation}
\label{Y'j}
Y_{p^i(k)} \ast X_0+(1-\omega)^{-1}(\Delta \circ \sigma^{-1})(Y_{p^i(k)}) \in \mathcal{B}^{ex}_{[0,p^i(k)]} \setminus \mathcal{B}^{ex}_{[0,p^i(k)-1]} \end{equation}
is a homogeneous prime element of $\mathcal{B}^{ex}_{[0,p^i(k)]}$. Sill by Theorem \ref{HP-CGL}, (\ref{Y'j}) is a scalar multiple of $Y'_{p^i(k)}$. Through direct computation,
\[ (\ref{Y'j})=X_0 \ast \sigma^{-1}(Y_{p^i(k)})+\omega(1-\omega)^{-1}(\Delta \circ \sigma^{-1})(Y_{p^i(k)}). \]
It is easy to see that $X_0 \ast Y_{p^i(k)}+\omega(1-\omega)^{-1}\Delta(Y_{p^i(k)})$ is a scalar multiple of (\ref{Y'j}). Since $\mathcal{B}^{ex}$ is symmetric, the degree of $\Delta(Y_{p^i(k)})$ is less than that of $Y_{p^i(k)}$. As the highest degree term of $Y_{p^i(k)}$ has coefficient 1 by definition, so does that of $X_0 \ast Y_{p^i(k)}+\omega(1-\omega)^{-1}\Delta(Y_{p^i(k)})$. Since the highest degree term of $Y'_{p^i(k)}$ also has coefficient 1 by definition,
\begin{equation}
\label{commute}
Y'_{p^i(k)}=X_0 \ast Y_{p^i(k)}+\omega(1-\omega)^{-1}\Delta(Y_{p^i(k)}).
\end{equation}

Case 2: $j \in [1,n]$ and $j \neq p^i(k)$ for any $i \in [0,m]$. By Theorem \ref{HP-CGL}, $Y_j \in \mathcal{B}^{ex}_{[0,j]} \setminus \mathcal{B}^{ex}_{[0,j-1]}$ is a homogeneous prime element of $\mathcal{B}^{ex}_{[0,j]}$. Since the highest degree terms of $Y'_j$ and $Y_j$ both have coefficients 1, $Y'_j=Y_j$.

Apply similar discussion to $B$. When $j=p^i(k)$ for some $i \in [0,m]$,
\[ y'_{p^i(k)}=x_0 \ast y_{p^i(k)}-\frac{\delta(y_{p^i(k)})}{\eta}. \]
When $j \in [1,n]$ and $j \neq p^i(k)$ for any $i \in [0,m]$, $y'_j=y_j$. Since $\mathcal{A}^{ex}$ and $\mathcal{A}$ share the same level sets, by 2) of Lemma \ref{one-one} and by 2) of Lemma \ref{qone-one}, $y_j$ and $Y_j$ have the same degree for $j \in [1,n]$. Clearly, $y'_0=x_0$ and $Y'_0=X_0$ have the same degree. When $j \in [1,n]$ and $j \neq p^i(k)$ for any $i \in [0,m]$, $y'_j=y_j$ and $Y'_j=y'_j$ have the same degree. Assume $j=p^i(k)$ for some $i \in [0,m]$. Since $B$ is symmetric, the degree of $\delta(y_{p^i(k)})$ is less than that of $y_{p^i(k)}$. It follows that the degree of $y'_{p^i(k)}$ is equal to the degree of $x_0y_{p^i(k)}$. Similarly, the degree of $Y'_{p^i(k)}$ is equal to the degree of $X_0 \ast Y_{p^i(k)}$. As $y_{p^i(k)}$ and $Y_{p^i(k)}$ have the same degree, $y'_{p^i(k)}$ and $Y'_{p^i(k)}$ have the same degree. Therefore, $y'_j$ and $Y'_j$ have the same degree for $j \in [0,n]$. By 2) of Lemma \ref{one-one} and by 2) of Lemma \ref{qone-one}, $B$ and $\mathcal{B}^{ex}$ share the same level sets.

Let $i \in [0,m]$. By the definition of $D^{(i)}$ and $\Delta$, one sees that $D^{(i)}$ is the distinguished element with respect to $\mathcal{A}^{ex}_{p^i(k)}[X_0;\sigma,\Delta]$. By Theorem \ref{HP-CGL},
\[ -(1-\omega)^{-1}(\Delta \circ \sigma^{-1})(Y_{p^i(k)})=Y_{p^{i}(k)} \ast D^{(i)}. \]
By part 2) of Lemma \ref{lie in}, $Y_{p^{i}(k)} \ast D^{(i)} \in \mathcal{A}_{p^{i}(k)}$. As $\omega$ is a power of $q$ and by 3) of Definition \ref{Q-CGL},
\[ \omega(1-\omega)^{-1}\Delta(Y_{p^{i}(k)}) \in \mathcal{A}_{p^{i}(k)}. \]
By (\ref{commute}), $Y'_{p^{i}(k)} \in \mathcal{B}$. When $j \in [1,n]$ and $j \neq p^i(k)$ for any $i \in [0,m]$, $Y'_j=Y_j \in \mathcal{B}$. Clearly, $Y'_0=X_0 \in \mathcal{B}$. Therefore, $Y_{\mathcal{B}^{ex}}$ lie in $\mathcal{B}$.

In conclusion, $\mathcal{B}=L[X_0;\sigma_0,\Delta_0][X_1;\sigma_1,\Delta_1]\cdots [X_n;\sigma_n,\Delta_n]_{(\lambda_0,\lambda_1,...,\lambda_n)}$ is a preferred quantization of $B$.

{\bf Proof of Theorem \ref{unique}}: When $\delta=0$, as $\mathcal{B}'$, $\mathcal{B}$ and $B$ share the same level sets, $\Delta'=0=\Delta$. Assume $\delta \neq 0$. By Lemma \ref{qd-M}, $\Delta'$ depends solely on the choice of $\Delta'(X_{p^m(k)})$, which is a monomial in $\Gamma_{q,p^m(k)-1}$. Since $B$ is the semiclassical limit of $\mathcal{B}'$,
\[ \frac{\Delta'(X_{p^m(k)})}{q-1}+(q-1)\mathcal{A}=\delta(x_{p^m(k)}). \]
Hence, $\Delta'(X_{p^m(k)})$ has to be a scalar multiple of $f(\delta(x_{p^m(k)}))$. Thus, $\Delta'=\epsilon \Delta$ for some $\epsilon \in \mathbb{K}$. By 3) of Definition \ref{per-qua} and by (\ref{commute}),
\[ \omega(1-\omega)^{-1}\Delta(Y_{p^m(k)}) \in \mathcal{A}_{p^{m}(k)}. \]
As $\omega$ is a power of $q$ and by 3) of Definition \ref{Q-CGL},
\[ \epsilon Y_{p^m(k)} \ast D^{(m)}=-(1-\omega)^{-1}(\Delta' \circ \sigma^{-1})(Y_{p^m(k)}) \in  \mathcal{A}_{p^{m}(k)}. \]
It follows from $\epsilon Y_{p^m(k)} \ast D^{(m)}=\epsilon Y_{p^m(k)} \ast  f(d^{(m)})$ that $\epsilon \in L$. As $\mathcal{B}'$ and $\mathcal{B}$ have the same semiclassical limit, one sees that $\epsilon|_{q=1}=1$.

\end{document}